\documentclass[12pt,a4paper]{amsart}
\usepackage[utf8]{inputenc}
\usepackage{amsmath}
\usepackage{amsfonts}
\usepackage{amssymb}
\usepackage{amsthm}
\usepackage{graphicx}
\usepackage{helvet}
\usepackage{hyperref}
\usepackage[margin=2.9cm]{geometry}

\newtheorem*{thm}{Theorem}
\newtheorem{theorem}{Theorem}
\newtheorem{corollary}{Corollary}
\newtheorem{proposition}{Proposition}

\newtheorem{remark}{Remark}

\newtheorem{lemma}{Lemma}

\begin{document}

	\author{Sumit Kumar }
	
	\title{Subconvexity bound for $\mathrm{GL(3)} \times \mathrm{GL(2)}$ $L$-functions in  $\mathrm{GL(2)}$ spectral aspect}
	
	\address{Sumit Kumar \newline  Stat-Math Unit, Indian Statistical Institute, 203 B.T. Road, Kolkata 700108, India; email: sumitve95@gmail.com}

	\date{\today}
	
	.
	
	\maketitle
	
	\section*{ Abstract}
	
	Let $\pi$ be a  Hecke-Maass cusp form for $\mathrm{SL(3, \mathbb{Z})}$  and $f$ be a holomorphic cusp form for $\mathrm{SL(2,\mathbb{Z})}$  of weight $k$ or a Hecke-Maass cusp form corresponding to the Laplacian eigenvalue $1/4+k^2$, $k\geq 1$,  for $\mathrm{SL(2,\mathbb{Z})}$. In this paper, we prove the following subconvexity bound
	\begin{align*}
		L\left({1}/{2},\  \pi \times f\right) \ll_{\pi,\,\epsilon}   k^{\frac{3}{2}-\frac{1}{51}+\epsilon}. 
	\end{align*}
	

	\section{ Introduction}
	A degree $d $  automorphic $L$-function $L(s, F)$ associated to an automorphic form $F$   is a Dirichlet series with an  Euler product of degree $d$ and   satisfying  some ``nice" analytic properties. In fact, it has a meromorphic continuation to the whole  complex plane $\mathbb{C}$ and its completed $L$-function satisfies a functional equation relating its value at $s$ to the value of the corresponding  dual $L$-function at  $1-s$.  One may apply  the Phragm\'en-Lindel\"{o}f principle together with the functional equation to get an upper bound
	 $L(1/2+it, F)  \ll_{d,\epsilon} (C(F,t))^{1/4+\epsilon}$, for any $\epsilon>0$,  on the critical line $ \Re s =1/2$.  Here $C(F,t)$ is a quantity, so-called the analytic conductor, which measures the complexity of the $L$-function and encapsulates the main parameters (level, spectral parameters, etc.) attached to the form $F$. The resulting bound is usually referred to as the convexity bound (or the trivial bound). It is conjectured, known as the Lindel\"{o}f Hypothesis, that the exponent $1/4$ can be reduced to $0$.  While the  Lindel\"{o}f Hypothesis is still out of reach, breaking the convexity bound,  i.e, reducing the exponent $1/4$ by  any small quantity,  known as the subconvexity problem,   is a challanging yet an   interesting problem.  \\

	For degree one $ L$-functions ($\zeta(s)$ and $L(s, \chi)$), such  estimates are known due to Weyl (\cite{Weyl}) and Hardy-Littlewood in the $t$-aspect and due to Burgess \cite{Burg}  in the  level aspect. For degree two $L$-functions, the  first subconvexity bound  was  achieved by  Good \cite{gl3gl2ood} in the $t$-aspect, by Duke-Friedlander-Iwaniec \cite{DFI}, \cite{DFI-2}, \cite{DFI-2.1} in the level aspect   and by Iwaniec \cite{iwa} in the spectral aspect. For  degree three $L$-functions attached to  self-dual forms, such estimates were first  obtained  by  Li \cite{Li}  in the $t$-aspect   in a groundbreaking work.  Li's work was generalised to all $\mathrm{GL(3)}$ forms by  Munshi  \cite{munshi1}, by  introducing  a novel delta method  which he also applied in resolving the  subconvexity  problem for $\mathrm{GL(3)}$ $L$-functions in the twist aspect  \cite{annals}. In  $\mathrm{GL(3)}$ spectral aspect, when  spectral  parameters  of a $\mathrm{GL(3)}$ form,  $\pi$ say,  are in generic position, subconvexity estimates  for $L(1/2, \pi)$ were obtained by  Blomer-Buttcane \cite{Blomer}. \\

	For higher degree $ L$-functions, the subconvexity problem becomes even more challenging,  and hence  it is still open except for a few particular cases of Rankin-Selberg convolution $L$-functions.  For  the Rankin-Selberg $L$-functions on $\mathrm{GL(2)} \times \mathrm{GL(2)}$,  subconvexity bounds are known due to Michel-Venkatesh \cite{MV} in the $t$-aspect,  Sarnak \cite{Sar}, and Lau-Liu-Ye \cite{LLY} in the spectral aspect, and   Kowalski-Michel-Vanderkam \cite{KMV}, Michel \cite{Mic} and  Harcos-Michel \cite{HM} in the level aspect. Some  impressive  subconvexity estimates were obtained by Bernstein-Reznikov \cite{BR} and Venkatesh \cite{Venk}  for  the Rankin-Selberg triple $L$-functions on $\mathrm{GL(2)}$. \\
	
	We will now discuss a few known results for  degree six Rankin-Selberg $L$-functions on  $\mathrm{GL(3)} \times \mathrm{GL(2)}$. To start with,  let  $\pi$ be a normalized Hecke-Maass cusp form of type $(\nu_{1},\nu_{2})$ for $\mathrm{SL(3, \mathbb{Z})}$. Let $f$ be a holomorphic cusp form with weight $k$ or Maass Hecke cusp form with the Laplace eigenvalue $1/4+k^2$  for $\mathrm{SL(2, \mathbb{Z})}$. The associated  Rankin-Selberg $L$-series is given by 
	\begin{align}\label{Lfunction}
		L(s,\pi \times f) = \, \mathop{\sum \sum}_{n,r \geq 1} \frac{\lambda_{\pi}(n,r) \, \lambda_{f}(n)}{(n r^2)^{s}}, \, \, \Re(s) >1.
	\end{align}
	In a pioneering work,  Li  \cite{Li}  studied  the above series  and  obtained subconvexity  for $L(1/2, \pi \times f)$ in the $\mathrm{GL(2)}$ spectral aspect as well as subconvexity for  $L(1/2+it, \pi)$ for a self-dual form $\pi$ (also mentioned above). Her main  theorem was the following:
	\begin{thm}[{X. Li.}]
		Let $\pi$ be a fixed self-dual Hecke-Maass cusp form for $ \mathrm{SL(3, \mathbb{Z})}$ and $u_j$ be an orthonormal basis of even  Hecke-Maass cusp form for $\mathrm{SL(3, \mathbb{Z})}$  corresponding to the Laplacian eigenvalue $1/4+t_j^2$ with $t_j \geq 0$; then for $\epsilon >0$, $T$ large and $T^{\frac{3}{8}+\epsilon} \leq M\leq T^{\frac{1}{2}}$, we have
		\begin{align*}
			\sideset{}{^\prime}\sum_{j}e^{-\frac{(t_j-T)^2}{M^2} } L\left(\frac{1}{2},\pi \times u_j\right)+\frac{1}{4\pi}\int_{-\infty}^{\infty}e^{-\frac{(t_j-T)^2}{M^2}}\left|L\left( \frac{1}{2}-it,\pi\right)\right|^2dt \ll_{\epsilon,\pi}T^{1+\epsilon}M,
		\end{align*} 
		where $\prime$ means summing over the orthonormal basis of even  Hecke-Maass cusp forms.
	\end{thm}
	As a corollary, she obtained  $	L\left({1}/{2},\pi \times u_j\right) \ll_{\epsilon, \pi} (1+|t_j|)^{\frac{3}{2}-\frac{1}{8}+\epsilon}.$
	She adapted Conrey-Iwaniec's moment method (see \cite{CI}) to prove the above  theorem. The fact  $ L\left({1}/{2},\pi \times u_j\right) \geq 0$ plays a crucial role in her approach. 
	Unfortunately, the above fact does not hold  for  non-self-dual $ \mathrm{GL(3)} $ forms, that is why she only dealt with the  self-dual forms. Recently,  Munshi  \cite{munshi12}, using  his delta method,  obtained  subconvexity for  $L(s,\pi \times f)$ in the $t$-aspect proving  the following result:
	\begin{align*}
		L\left({1}/{2}+it,\pi \times f\right) \ll_{\epsilon,f,\pi} (1+|t|)^{\frac{3}{2}-\frac{1}{51}+\epsilon}.
	\end{align*}
	His method is insensitive to  self-duality of the $ \mathrm{GL(3)} $ forms.  Thus he could obtain the above result for any $ \mathrm{GL(3)} $ form. Using a similar approach,  Sharma \cite{prahlad} and the author, Mallesham and  Singh \cite{KMS} proved  subconvexity bounds in the $\mathrm{twist}$ and the $ \mathrm{GL(3)} $ spectral aspect (in some cases), respectively. It is natural to ask a similar question in other aspects as well. As suggested by  Munshi, we vary the $\mathrm{GL(2)}$ form  and    establish  subconvexity  for $L(1/2, \pi \times f)$ in the $ \mathrm{GL(2)} $ spectral aspect. Our main theorem  is the following:
	\begin{theorem}\label{Main}
		Let $\pi$ be a fixed Hecke-Maass cusp form for $ \mathrm{SL(3, \mathbb{Z})}$   and $f$ be a holomorphic cusp form with weight $k$ or a Hecke-Maass cusp form corresponding to the Laplacian eigenvalue $1/4+k^2$,  $k \geq 1$,   for $ \mathrm{SL(2, \mathbb{Z})}$ . Then, for any $\epsilon >0$,  we have 
		\begin{align*}
			L\left({1}/{2},\, \pi \times f\right) \ll_{\pi,\epsilon}   k^{\frac{3}{2}-\frac{1}{51}+\epsilon}.
		\end{align*}
	\end{theorem}
	\begin{remark} 
		$\mathrm{(1)}$ We follow Munshi  \cite{munshi12} to prove the above theorem. As  mentioned before, the method is insensitive to  self-duality of the  $ \mathrm{GL(3)} $ forms, we also obtain our result for any $ \mathrm{GL(3)} $ form. Thus we generalise   Li's main result \cite{Li}. Although our bound is weaker than her's, it yields a subconvexity.\\
		\noindent
		$\mathrm{(2)}$ Our arguments in the proof work for both Maass and holomorphic forms. For the exposition of the method, we  will give details for holomorphic forms only.\\
		\noindent
		$\mathrm{(3)}$ Our method also works for any fixed central value $1/2+it$. In this case, the implied constant will depend polynomially on $t$.  For simplicity, we take $t=0$ in the proof.
		
	\end{remark}
	If we take $\pi$ to be the minimal Eisenstein series with Langlands parameters ($\alpha_{1}$, $\alpha_{2}$, $\alpha_{3}$) for $\mathrm{SL(3,\mathbb{Z})}$ in \eqref{Lfunction},  we observe that (see \cite[p.\,314]{gold}) 
	\begin{align*}
		L(s,\pi)=\prod_p\prod_{i=1}^{3}\left(1-p^{\alpha_{i}-s}\right)^{-1}=\zeta(s-\alpha_{1})\zeta(s-\alpha_{2})\zeta(s-\alpha_{3}). 
	\end{align*}
	It is also well-known that 
	\begin{align*}
		L(s,f)=\sum_{n=1}^{\infty}\frac{\lambda_f(n)}{n^s}=\prod_p\prod_{j=1}^{2}\left(1-\beta_{p,j}p^{-s}\right)^{-1},
	\end{align*}
	where $\beta_{p,1}\beta_{p,2}=1$, $\beta_{p,1}+\beta_{p,2}=\lambda_f(p)$, and $\lambda_f(n)$ denote the normalised  Fourier coefficients of $f$.  Using  the Rankin-Selberg  theory (see \cite[p.\,379]{gold}), we get 
	\begin{align*}
		L(s,\pi \times f)&= \prod_p\prod_{i=1}^{3}\left(1-\beta_{p,1}\, p^{\alpha_{i}-s}\right)^{-1}\left(1-\beta_{p,2}\, p^{\alpha_{i}-s}\right)^{-1} \\
		&= L(s-\alpha_{1},f)L(s-\alpha_{2},f)L(s-\alpha_{3},f).
	\end{align*}
	Our method  also fits  for the above $L$-functions as well. Hence we also obtain the following result:
	\begin{theorem}
		Let	$f$ be a holomorphic cusp form with weight $k$ or Hecke-Maass cusp form corresponding to the Laplacian eigenvalue $1/4+k^2$, $k \geq 1$,  for $\mathrm{SL(2,\mathbb{Z})}$. Then, for $\epsilon >0$, we have 
		\begin{align*}
			L\left( {1}/{2},\,  f\right) \ll_{\epsilon}   k^{\frac{1}{2}-\frac{1}{153}+\epsilon}.
		\end{align*}
	\end{theorem}
	
	We end the introduction by commenting on the method. As a first step, we use the `conductor lowering' trick as a device to   separate the oscillations using the delta method due to   Duke-Friedlander-Iwaniec (DFI) (see \cite{DFI}). We now apply the  $ \mathrm{GL(3)} $ and $ \mathrm{GL(2)} $ Voronoi  formulae.  A crucial observation, which  was  also present  in  \cite{KMS}, \cite{munshi12} and  \cite{prahlad},  that the $ \mathrm{GL(2)} $ and  $ \mathrm{GL(3)} $ Voronoi formulae together tranform the Ramanujan sum 
	$$ \sideset{}{^\star}{\sum}_{a \, {\rm mod} \, q} \, e \left(\frac{(n-m)a}{q}\right),$$
	arising from  the DFI delta method, into 
	$$ \sideset{}{^\star}{\sum}_{a \, {\rm mod} \, q} S(\overline{a},n;q) e \left(\frac{\overline{a}m}{q}\right),$$
	which boils down to an  additive character $qe(\overline{m}n/q)$  with respect to $n$,  plays a vital role  to prove our main  theorem.  In fact, we save $\sqrt{q}$ extra after  applying the Poisson summation formula in the $n$-variable due to the additive character (see subsection \ref{sketch}).   The  analysis of the integral transforms (see Section \ref{estimates for int}) is one of technical  inputs of the article.  \\
	\noindent 
	
	\paragraph*{\bf{Notation}} Throughout the paper, $e(x)$ means $e^{2\pi i x}$.  By  negligibly small we mean  $O(k^{-A})$ for any large postive constant  $A >0$. In particular, we  take $A=2020$.  The letter $\epsilon$ denotes arbitrarily small constant, not necessarily the same at different occurrences. The notation $\alpha \ll A$ will mean that for any $\epsilon >0$, there is a constant $c$ such that $|\alpha| \leq cAk^{\epsilon}$. We also ignore the dependence of the constant on $\pi$ and $\epsilon$, whenever it occurs.  By $\alpha \asymp A$, we mean that $k^{-\epsilon}A \leq \alpha \leq k^{\epsilon}A$, also $\alpha \sim A$ means $A\leq \alpha < 2A$. For absolute explicit constant, we will write $c$ or $c_i$ in the whole paper. \\
	\noindent
	
	\paragraph*{\bf{Acknowledgements.}} This work is a part of the author's Ph.D. thesis and he is  grateful to his advisor Prof. Ritabrata Munshi for suggesting the problem, sharing his beautiful ideas, explaining his ingenious method, and his kind support throughout the work. He  would also like  to thank Prof. Satadal Ganguly, Saurabh Kumar Singh and Kummari Mallesham for their encouragement and  constant support and  Stat-Math Unit, Indian Statistical Institute, Kolkata for the excellent research environment.  Finally, the author would  like to thank the referees for their valuable  suggestions and comments which really helped  in improving  the  presentation of the article. 
	
	\section{ Preliminaries}
	In this section, we will recall some known results which we need in the proof. 
	\subsection{ Holomorphic cusp  forms on $\textrm{GL(2)}$}
	Let $f$ be a holomorphic Hecke eigenform  of weight $k$ for the full modular group $\mathrm{SL(2,\mathbb{Z})}$. The Fourier expansion of $f$ at the cusp  $\infty$ is given by 
	$$f(z) = \sum_{n=1}^{\infty} \, \lambda_{f}(n) \, n^{(k-1)/2} \, e(nz),\ \ \ z \in \mathbb{H}.$$
	We assume that $f$ is normalised so that $\lambda_{f}(1)=1$.  We have   the    well-known  Deligne bound  $	\vert \lambda_{f}(n) \vert \leq d(n), n\geq 1,$ 
%
	where $d(n)$ is the divisor function.  However, in our proof, we only need the Ramanujan bound on average: 
	\begin{align}\label{RS for holomorphic}
		\sum_{n \leq X}|\lambda_{f}(n)|^2 \ll_\epsilon X^{1+\epsilon}, 
	\end{align}
	for any $\epsilon>0$.   We now recall the Voronoi summation formula for the  form $f$  which will be crucially used in our proof. 
	
	\begin{lemma} \label{gl2voronoi}
		Let $\lambda_{f}(n)$ be as above and $g$ be a smooth, compactly supported function on $(0, \infty)$. Let $a$, $q \in \mathbb{Z}$ with $(a,q)=1$. Then we have
		$$\sum_{n=1}^{\infty} \lambda_{f}(n) \, e\left(\frac{an}{q}\right)g(n) = \frac{ 1}{q} \sum_{n=1}^{\infty} \lambda_{f}(n) \, e\left(-\frac{d n}{q}\right)\, h\left(\frac{n}{q^2} \right),$$
		where $ad \equiv 1 (\mathrm{mod} \, q)$ and 
		$$h(y) = 2\pi i^k\int_{0}^{\infty} g(x) \, J_{k-1} \left({4 \pi \sqrt{xy}}\right) \, \mathrm{d}x,$$
		where $J_{k-1}$ is the usual $J$-Bessel function of order $k-1$. 
	\end{lemma}
	\begin{proof}
		See \cite[Theorem A.4]{KMV}.
	\end{proof}

	\subsection{Maass cusp forms for $\textrm{GL(2)}$} Let $f$ be a Hecke-Maass eigenform  for $\mathrm{SL(2, \mathbb{Z})}$ with Laplace eigenvalue $1/4 + \nu^2$, $\nu>0$. The Fourier series expansion of $f$ at  the cup  $\infty$ is given by 
	\[
	f(z)= \sqrt{y} \sum_{n \neq 0} \lambda_f(n) K_{ i \nu} (2 \pi |n|y) e(nx), 
	\] 
	where $ K_{ i \nu}(y)$ is the   Bessel function of  the third  kind and $f$ is normalized so that  $\lambda_{f}(1)=1$.   The  Ramanujan-Petersson conjecture, which assert that 
	 $|\lambda_f(n)| \leq  d(n)$ is not known yet. 
	   However, we do not need of such individual bound for our proof. Rather, the  following Ramanujan bound on average  (see  \cite[Lemma 1]{iwa})
	 	\begin{align}\label{RS bound}
	 	\sum_{1\leq n \leq X} \left| \lambda_f(n) \right|^2 \ll_{ \epsilon} \nu^\epsilon X^{1+\epsilon},
	 \end{align} 
 for any $\epsilon>0$,   is sufficient for our purpose. 
 We also  have the  following   Voronoi  summation formula for  the Maass cusp forms, which is similar to the case of  holomorphic cups forms.  
	\begin{lemma}
		Let $\lambda_{f}(n)$ be as above   and $g$ be a smooth, compactly supported function on $(0, \infty)$. Let $a$, $q \in \mathbb{Z}$ with $(a,q)=1$. Then we have
		\begin{align*}
			\sum_{n=1}^{\infty} \lambda_{f}(n) \, e\left(\frac{an}{q}\right)g(n)=\frac{1}{q}\sum_{\pm}\sum_{n=1}^{\infty} {\lambda_f( n)}e\left(\mp \frac{dn}{q}\right)H^{\pm}\left(\frac{n}{q^2}\right),
		\end{align*}
		where $ad \equiv 1 \, \mathrm{mod} \, q$ and  
		\begin{align*}
			H^-(y)=\frac{-\pi}{\sin(\pi i\nu)}  \int_{0}^{\infty} g(x) \, \{J_{2i\nu}-J_{-2i\nu}\} \left({4 \pi \sqrt{xy}}\right) \, \mathrm{d}x,
		\end{align*}
			\begin{align*}
			H^+(y)={4\epsilon_f\cosh(\pi \nu)}  \int_{0}^{\infty} g(x) \,K_{2i\nu}\left({4 \pi \sqrt{xy}}\right) \, \mathrm{d}x.
		\end{align*}
	Here $\epsilon_f$ is the eigenvalue of $f$ under the reflection operator. 
	\end{lemma}
	\begin{proof}
		See \cite[Theorem A.4]{KMV}.
	\end{proof}

	\subsection{Automorphic  forms on $\textrm{GL(3)}$}
 This   section, except for the notations, is taken from \cite{Li}. Let $\pi$ be a Hecke-Maass cusp form of type $(\nu_{1}, \nu_{2})$ for $\mathrm{SL(3, \mathbb{Z})}$. Let $\lambda_{\pi}(n,r)$ denote the normalised Fourier coefficients of $\pi$. Let 
	$${\alpha}_{1} = - \nu_{1} - 2 \nu_{2}+1, \, {\alpha}_{2} = - \nu_{1}+ \nu_{2}  \ \mathrm{and} \ {\alpha}_{3} = 2 \nu_{1}+ \nu_{2}-1$$ 
	be the spectral parameters for $\pi$ (see  \cite{gold}).
	Let $g$ be a compactly supported smooth function on  $ (0, \infty )$ and 
	$$\tilde{g}(s) = \int_{0}^{\infty} g(x) x^{s-1} \mathrm{d}x$$ be its Mellin transform. For $\ell= 0$ and $1$, we define
	\begin{equation}
		\gamma_{\ell}(s) :=  \frac{\pi^{-3s-\frac{3}{2}}}{2} \, \prod_{i=1}^{3} \frac{\Gamma\left(\frac{1+s+{\alpha}_{i}+ \ell}{2}\right)}{\Gamma\left(\frac{-s-{\alpha}_{i}+ \ell}{2}\right)}.
	\end{equation}
	Set $\gamma_{\pm}(s) = \gamma_{0}(s) \mp \gamma_{1}(s)$ and let 
	\begin{align}\label{gl3 integral transform}
		G_{\pm}(y) = \frac{1}{2 \pi i} \int_{(\sigma)} y^{-s} \, \gamma_{\pm}(s) \, \tilde{g}(-s) \, \mathrm{d}s,
	\end{align}
	where $\sigma > -1 + \max \{-\Re({\alpha}_{1}), -\Re({ \alpha}_{2}), -\Re({\alpha}_{3})\}$. With the aid of the above terminology, we now state the $ \mathrm{GL(3)} $ Voronoi summation formula.
	\begin{lemma} \label{gl3voronoi}
		Let $g(x)$ and  $\lambda_{\pi}(n,r)$ be as above. Let $a, \, q \in \mathbb{Z}$ with $q\geq 1, (a,q)=1,$ and  $a\bar{a} \equiv 1(\mathrm{mod} \ q)$. Then we have
		\begin{align*} \label{GL3-Voro}
			\sum_{n=1}^{\infty} \lambda_{\pi}(n,r) e\left(\frac{an}{q}\right) g(n) 
			=q  \sum_{\pm} \sum_{n_{1}|qr} \sum_{n_{2}=1}^{\infty}  \frac{\lambda_{\pi}(n_1,n_2)}{n_{1} n_{2}} S\left(r \bar{a}, \pm n_{2}; qr/n_{1}\right) G_{\pm} \left(\frac{n_{1}^2 n_{2}}{q^3 r}\right),
		\end{align*} 
		where  $S(a,b;q)$ is the  Kloosterman sum which is defined as follows:
		$$S(a,b;q) = \sideset{}{^\star}{\sum}_{x \,{\rm mod}\,  q} e\left(\frac{ax+b\bar{x}}{q}\right).$$
	\end{lemma}
	\begin{proof}
		See \cite{miller-schmid}. 
	\end{proof}
	We  need to extract the  oscillations of the  integral transform $G_{\pm}$.  To this end, we have the following lemma:
	\begin{lemma} \label{GL3oscilation}
		Let $G_{\pm}(x)$ be as above,  and  $g(x) \in C_c^{\infty}(X,2X)$. Then for any fixed integer $K \geq 1$ and $xX \gg 1$, we have
		\begin{equation*}
			G_{\pm}(x)=  x \int_{0}^{\infty} g(y) \sum_{j=1}^{K} \frac{c_{j}({\pm}) e\left(3 (xy)^{1/3} \right) + d_{j}({\pm}) e\left(-3 (xy)^{1/3} \right)}{\left( xy\right)^{j/3}} \, \mathrm{d} y + O \left((xX)^{\frac{-K+5}{3}}\right),
		\end{equation*}
		where $c_{j}(\pm)$ and $d_{j}(\pm)$ are some   absolute constants depending on $\alpha_{i}$'s, for $i=1, 2, 3$.  
	\end{lemma}
	\begin{proof}
		See   \cite[Lemma 6.1]{Li*}.
	\end{proof}
	The following lemma is the  well-known  Ramanujan bound on average,
	\begin{lemma} \label{ramanubound}
		We have 
		\begin{align}\label{rankin selberg}
			\mathop{\sum \sum}_{n_{1}^{2} n_{2} \leq x} \vert \lambda_{\pi}(n_{1},n_{2})\vert ^{2} \ll_{\pi} \, x,
		\end{align}
		where the implied constant depends on the form $\pi$.
	\end{lemma}
	\begin{proof}
		For the proof, we refer to Goldfeld's book \cite{gold}.
	\end{proof}

	\subsection{The delta method} \label{circlemethod}
	Let $\delta: \mathbb{Z} \to \{0,1\}$ be defined by
	\[
	\delta(n)= \begin{cases}
		1 \quad \ \  \text{if} \, \  n=0; \\
		0 \quad   \ \   $\textrm{otherwise}$.
	\end{cases}
	\]
	The above function can be used to separate the oscillations involved in a sum,   $\sum_{n \sim X}a(n)\, b(n)$, say.    Furthermore, we seek a `nice' Fourier expansion of $\delta(n)$. We mention here an expansion for $\delta(n)$  due to Duke, Friedlander and Iwaniec (see \cite[Chapter 20]{i+k}). Let $L\geq 1$ be a large real number. For $n \in [-2L,2L]$, we have
	\begin{align} \label{delta first expansion}
		\delta(n)= \frac{1}{Q} \sum_{1 \leq q \leq Q} \frac{1}{q} \, \sideset{}{^\star}{\sum}_{a \, {\rm mod} \, q} \, e \left(\frac{na}{q}\right) \int_{\mathbb{R}} g(q,x) \,  e\left(\frac{nx}{qQ}\right) \, \mathrm{d}x,
	\end{align}
	where  $Q=2L^{1/2}$. The $\star$ on the sum indicates that the sum over $a$ is restricted by the condition $(a,q)=1$. The function $g$ is the only part in the above formula which is not explicitly given. Nevertheless, we only need the following  properties of $g$ in our analysis.  For any $B >1$, we have (see \cite[p. 5-6]{munshi12})
	\begin{align} \label{g properties}
		&1. \ g(q,x)=1+h(q,x), \quad \text{with} \ \ \  h(q,x)=O \left(\frac{Q}{q} \left(\frac{q}{Q}+|x|\right)^{B}\right).  \\
		&2. \  x^j \frac{\partial ^j}{\partial x^j}g(q,x) \ll \log Q \min \left\lbrace \frac{Q}{q}, \frac{1}{|x|}\right\rbrace, \ j\geq 1. \notag  \\
		&3. \  g(q,x) \ll |x|^{-B}.  \notag \\
		&4. \ 	\int_{\mathbb{R}}|g(q,x)|+|g(q,x)|^2\,  \mathrm{d}x \ll Q^{\epsilon}. \notag
	\end{align}
	Using the third property we observe that the effective range of the $x$-integral in \eqref{delta first expansion}  is $[-Q^{\epsilon},\,Q^{\epsilon}]$. We record the above observations in the following lemma.
	\begin{lemma}\label{deltasymbol}
		Let $\delta$ be as above.  Let $L\geq 1$ be a large parameter. Then, for $n \in [-2L,2L]$, we have
		\begin{equation*} 
			\delta(n)= \frac{1}{Q} \sum_{1 \leq q \leq Q} \frac{1}{q} \, \sideset{}{^\star}{\sum}_{a \, {\rm mod} \, q} \, e \left(\frac{na}{q}\right) \int_{\mathbb{R}}W(x/Q^\epsilon) g(q,x) \,  e\left(\frac{nx}{qQ}\right) \, \mathrm{d}x+O(L^{-2020}),
		\end{equation*}
		where $Q=2L^{1/2}$, $g$ is a function satisfying \eqref{g properties} and  $W(x)$ is a non-negative smooth bump function supported in $[-2,\,2]$, with $W(x)=1$ for $x \in [-1,\,1] $ and  $W^{(j)}(x)\ll_j 1$, for $j\geq 0$.
	\end{lemma}
	\begin{proof}
		See \cite[Chapter 20]{i+k} and \cite[Lemma 15]{BingH}.
	\end{proof}

	\subsection{Bessel function} 
	In this subsection, we will recall some well-known expansions  of the Bessel functions of the first kind. For $k \geq 2$ an integer, let $J_{k-1}(x)$ be the Bessel function of the first kind  and of  order $k-1$, which is defined as 
	\begin{align}\label{Bessel defi}
		J_{k-1}(x)=\frac{1}{2\pi}\int_{-\pi}^{\pi}e\left( \frac{(k-1)\tau-x\sin \tau}{2\pi}\right)\textrm{d}\tau,
	\end{align}
	for any $x \in \mathbb{R}$.
	In the  analysis of integral transforms, we require  a uniform asymptotic  expansion of $J_{k-1}(x)$ for large values of $k$ and $x$.  The following lemma provides  one such   asymptotic  expansion.
	\begin{lemma}\label{langer series}
		Let  $x \geq (k-1)^{1+\epsilon/2}$ be a positive real number. Then, as $k  \rightarrow \infty$, we have 
		\begin{align}\label{langer exp}
			J_{k-1}(x)=&\left(\frac{2}{\pi (k-1)w}\right)^{1/2}\left[ \cos\left((k-1)(w-\tan^{-1}w)-\frac{\pi}{4}\right)\sum_{j=0}^{\infty}\frac{P_j\left(\frac{1}{w-\tan^{-1}w}\right)}{(k-1)^{j}}\right]  \\ 
			&+\left(\frac{2}{\pi (k-1)w}\right)^{1/2}\left[ \sin\left((k-1)(w-\tan^{-1}w)-\frac{\pi}{4}\right)\sum_{j=1}^{\infty}\frac{P_j\left(\frac{1}{w-\tan^{-1}w}\right)}{(k-1)^{j}}\right]\notag, 
		\end{align}
		where  $$w=\left(\frac{x^2}{(k-1)^2}-1\right)^{1/2},$$ and $P_j$ is  a polynomial of  the degree $j$ with coefficients which are bounded functions of $k-1$ and $\log(x/(k-1))$ with $P_0 \equiv 1$. 
	\end{lemma}
	\begin{proof}
		Let $x=(k-1) \sec \beta $, with $0<\beta<\pi/2$.  Thus, as $x \geq (k-1)^{1+\epsilon/2}$, we have $\sec \beta \geq (k-1)^{\epsilon/2}$ and  $$\xi:=(k-1)(\tan \beta -\beta) \geq (k-1)(\sqrt{(k-1)^{\epsilon} -1}-\pi/2).$$ 
		Thus,   on using  formula (63) on  page 58 of \cite{Langer},  we get 
		\begin{align*}
			J_{k-1}((k-1)\sec \beta)=&\left(\frac{2}{\pi (k-1)\tan \beta}\right)^{1/2}\left[ \cos f_1(\beta)\sum_{j=0}^{\infty}\frac{P_j\left(\frac{1}{\tan \beta-\beta}\right)}{(k-1)^{j}}\right] \\ 
			&+\left(\frac{2}{\pi (k-1) \tan \beta}\right)^{1/2}\left[ \sin f_1(\beta)\sum_{j=1}^{\infty}\frac{P_j\left(\frac{1}{\tan \beta-\beta}\right)}{(k-1)^{j}}\right], 
		\end{align*}
		where 	$f_1(\beta)=(k-1)(\tan \beta-\beta)-{\pi}/{4}$, and $P_j$ represents a polynomial of the degree $j$ with coefficients which are bounded functions of $k-1$ and $\log \sec \beta$ with $P_0 \equiv 1$. Now substituting $(k-1)\sec \beta=x$ and $\tan \beta=w$, we get the lemma. 
	\end{proof}
	The expansion \eqref{langer exp}  can be  truncated  at any stage to get  
	\begin{corollary}\label{Langer 2}
		Under the  assumptions  of  Lemma \ref{langer series},  we have
		\begin{align*} 
			J_{k-1}(x)=\sum_{\pm}\sum_{j=0}^{2019}\frac{e\left(\pm \frac{(k-1)(w-\tan^{-1}w)}{2\pi}\right)P_j\left(\frac{1}{w-\tan^{-1}w}\right)}{ \sqrt{\pi} w^{1/2}(k-1)^{j+1/2}} +O\left(\frac{1}{k^{2020}}\right).
		\end{align*}
	\end{corollary}
	\begin{proof}
		The statement follows directly from Lemma \ref{langer series}. 
	\end{proof}
	For $0 < x \leq (k-1)^{1-\epsilon/2}$,  we have the following lemma.
	\begin{lemma}\label{bessel for x <k}
		Let $x=(k-1)z$ with  $0 < z \ll (k-1)^{-\epsilon/2}$. Then, as   $k \rightarrow \infty$, we have 
		\begin{align*}
			J_{k-1}(x)\ll \exp\{-(k-1)/6\}.
		\end{align*}
		\begin{proof}
			By   Lemma 4.2 of \cite{rankin}, we have 
			$$|J_{k-1}((k-1)z)|  \leq A_1(k-1)^{-1/2}(1-z^2)^{-1/4}\exp\left\{-\frac{1}{3}(k-1)(1-z^2)^{3/2} \right\},$$
			for $0<z \leq \sqrt{1-\frac{1}{(k-1)^{2/3}}}$, $k \geq 16$ and    some absolute constant $A_1$. Note that, by assumption, we have $z \leq (k-1)^{-\epsilon}$.  Thus, $1-z^2 \geq 1/2^{2/3}$ as $k \rightarrow \infty$, and  we get 
			$$|J_{k-1}((k-1)z)|  \leq A_1 2^{1/6}\exp\left\{-\frac{1}{6}(k-1) \right\}.$$
			Hence 	the lemma follows. 
		\end{proof}
	\end{lemma}

	\subsection{Stationary phase analysis.} In this subsection, we will recall some facts about the exponential integrals of the form
	$$I= \int_{a}^bg(x)e(f(x))\mathrm{d}x,$$
	where $f$ and $g$ are smooth real valued functions on $[a,b]$.

	\begin{lemma}\label{derivative bound}
		Let $I$, $f$ and $g$ be as above. Let $V(g)$ denotes the total variation of $g(x)$ on $[a,b]$ plus the  maximum modulus of $g(x)$ on $[a,b]$.  Then, if $f^{\prime} $ is  monotone and $|f^{\prime}(x)|\geq \mu_1>0$ for $x \in [a,b]$, we have  $	{I} \ll { V(g)}/{\mu_1}.$  For $r > 1$, let $| f^{(r)}(x)|\geq \mu_r>0$. Then we have 	${I} \ll_r {V(g)}/{\mu_r^{1/r}}$.  Moreover, let $f^{\prime}(x)\geq B$ and $f^{(j)}(x) \ll B^{1+\epsilon}$ for $j \geq 2$ together with $\textrm{Supp}(g) \subset (a,b)$  and $g^{(j)}(x) \ll_{a,b,j} 1$. Then we have
	\begin{align}
		I \ll_{a,b,j,\epsilon}B^{-j+\epsilon}. \notag
	\end{align}
	\end{lemma}
	\begin{proof}
		See \cite[Subsection 2.2]{munshi1} and \cite[Lemma 5.1.4]{Huxley}.
	\end{proof}
We apply the above lemma  for $r=1$ whenever the phase function $f$ does not have any stationary point. We will also apply it for $r=2$ and $3$.  In case there is a unique stationary point, we use the following stationary phase expansion.
	\begin{lemma} \label{stationaryphase}
		Let $I$, $f$ and $g$ be as above. Let $0 < \delta < 1/10$, $X$, $Y$, $U$, $Q>0$, $Z:= Q+X+Y+b-a+1$, and assume that
		$$ Y \geq Z^{3 \delta}, \, b-a \geq U \geq \frac{Q Z^{\frac{\delta}{2}}}{\sqrt{Y}}.$$
		Further, assume that $g$  satisfies  
		$$g^{(j)}(x) \ll_{j} \frac{X}{U^{j}}, \, \,\ \  \, \text{for} \,\,  j=0,\,1,\,2,\ldots$$  
		Suppose that there exists a unique $x_{0} \in [a,b]$ such that $f^{\prime}(x_{0})=0$, and the function $f$ satisfies
		$$f^{\prime \prime}(x) \gg \frac{Y}{Q^2}, \ \ f^{(j)}(x) \ll_{j} \frac{Y}{Q^{j}}, \, \,\ \ \, \, \text{for} \ \ j=1,\,2,\,3,\ldots.$$
		Then we have
		\begin{align*}
			I &= \frac{e{ (f(x_{0}))}}{\sqrt{f^{\prime \prime}(x_{0})}} \, \sum_{n=0}^{3 \delta^{-1}A}  p_{n}(x_{0}) + O_{A,\delta}\left( Z^{-A}\right), \\ 
			p_{n}(x_{0}) &= \frac{ e^{\pi i/4}}{n!} \left(\frac{i}{2 f^{\prime \prime}(x_{0})}\right)^{n} G^{(2n)}(x_{0}),
		\end{align*}
		where $A>0$ is arbitrary, and 
		$$ G(x)=g(x) e( F(x)),\ \  \ \ \,\ \ \ F(x)= f(x)-f(x_{0})-\frac{1}{2} f^{\prime \prime}(x_{0})(x-x_{0})^2.$$
		Furthermore, each  $p_{n}$ is a rational function in $f^{\prime}, f^{\prime \prime}, \ldots,$ satisfying 
		$$\frac{d^{j}}{dx_{0}^{j}} p_{n}(x_{0}) \ll_{j,n} X \left(\frac{1}{U^{j}}+ \frac{1}{Q^{j}}\right) \left( \left(\frac{U^2 Y}{Q^2}\right)^{-n} + Y^{-\frac{n}{3}}\right).$$
	\end{lemma}
	\begin{proof}
		See \cite[Lemma 8.1]{BKY}.
	\end{proof}

\section{ \bf The set-up and outline of  proof}
Let $\pi$ and $f$ be  defined as in Theorem \ref{Main}. Let $\lambda_{\pi}(n,r)$ denote the normalised Fourier coefficients of the  form $\pi$ (see  \cite[Chapter 6]{gold}) and let $\lambda_{f}(n)$ denote the normalised Fourier coefficients of the form $f$ (see \cite[Chapter 14]{i+k}). We are interested in  analyzing  the Rankin-Selberg $L$-series $L(s,\pi \times f)$ (defined in \eqref{Lfunction}) attached to $\pi$ and $f$ at the central point $1/2$. To study $L \left( {1}/{2}, \pi \times f \right)$, we first express it as a weighted Dirichlet series.  
\begin{lemma}\label{AFE}
	Let $\theta$ be a positive real number such that $ 0< \theta < 3/2$. Then, as   $k \rightarrow \infty $, we have
	\begin{align} \label{aproxi} 
		L \left(1/2, \pi \times f \right) \ll k^{\epsilon}\sup_{r \leq k^{\theta}} \sup_{\frac{k^{3-\theta}}{r^2}\leq N\leq \frac{k^{3+\epsilon}}{r^2}} \frac{\left|S_r(N)\right|}{N^{1/2}} + k^{(3-\theta)/2+\epsilon},
	\end{align}
	where $S_r(N)$ is a sum of the form
	\begin{align} \label{s(n)-sum}
		S_r(N): = \mathop{\sum }_{n=1}^{\infty} \lambda_{\pi}(n,r) \lambda_{f}(n)  V \left(\frac{n}{N}\right),
	\end{align}
	for some smooth function $V$ supported in $[1,2]$, satisfying $V^{(j)}(x) \ll_{j} 1$ for  $j \geq 0$ and normalised so that $\int V(y) \mathrm{d}y=1$.
\end{lemma} 
\begin{proof}
Proof follows using the template given in   \cite[Theorem 5.3]{i+k}. See also  \cite[p.\,1546-1547]{munshi12}.  
	\end{proof}
\begin{remark}
	Upon estimating $S_r(N)$  using Cauchy's inequality and the Ramanujan bound on average (see  \eqref{RS for holomorphic}, \eqref{RS bound}, \eqref{rankin selberg}), we see that  $L\left({1}/{2}, \pi \times f\right)  \ll_{\pi, \epsilon} k^{3/2+\epsilon}$. Hence, to establish subconvexity, we need to get some cancellations in the sum $S_r(N)$ for $N$, roughly, of size $k^3$. To this end,  we will analyze  $S_r(N)$ in the rest of the paper. 
\end{remark}
	\subsection{An application of the delta method.}  As a first step, following Munshi \cite{munshi12}, we separate the  oscillatory terms  $\lambda_\pi(n,r)$ and $\lambda_{f}(n)$ involved  in the sum $S_r(N)$.  We  use the delta method of Duke, Friedlander and Iwaniec as a device to separate these terms. We also  apply  the conductor lowering trick  introduced by Munshi in \cite{munshi1}.  For this purpose, we introduce an extra $t$-integral. In fact, we express  $S_r(N)$  as
	\begin{align} \label{congtrick}
	&  \frac{1}{T}\int_{\mathbb{R}}V\left(\frac{t}{T}\right)\mathop{\sum \sum}_{\substack{n,\,  m=1 \\ n = m } }^{\infty} \lambda_{\pi}(n,r) \lambda_f(m) \left(\frac{n}{m}\right)^{it}  V \left(\frac{n}{N}\right) U \left(\frac{m}{N}\right)dt   \\
		=& \frac{1}{T}\int_{\mathbb{R}}V\left(\frac{t}{T}\right)\mathop{\sum \sum}_{\substack{n,\,  m=1  } }^{\infty}\delta(n-m) \lambda_{\pi}(n,r) \lambda_f(m) \left(\frac{n}{m}\right)^{it}  V \left(\frac{n}{N}\right) U \left(\frac{m}{N}\right)dt \notag,
	\end{align} 
	where $k^{\epsilon} <T<k^{1-\epsilon}$  is a parameter of the form $k^{1-\eta}$, for $\eta>0$, which will be chosen later optimally, and $U$ is a smooth function supported in $[1/2,5/2]$ with $ U(x)=1 \, \text{for} \, x \in [1,2]$, and $U^{(j)}(x) \ll_{j} 1$ for any integer $j \geq 0$. Consider the $t$-integral 
	\begin{align*}
		\frac{1}{T} \int_{\mathbb{R}} \, V\left(\frac{t}{T}\right) \left(\frac{m}{n}\right)^{i t} \, \mathrm{d}t.
	\end{align*}
	On applying integration by parts repeatedly, we observe that the above integral is negligibly small unless $\vert n-m\vert \ll k^{\epsilon}N/T$.  Thus the $t$-integral reduces the size of the equation $n=m$. Thus, on  applying  Lemma \ref{deltasymbol} to \eqref{congtrick} with $	L=k^{\epsilon}{N}/{T},$ and $ Q=k^{\epsilon}\sqrt{{N}/{T}},$
	we see that $S_r(N)$ is transformed into 
	\begin{align} \label{SN after dfi}
		S_r(N)=&\frac{1}{QT}\int_{\mathbb{R}}W(x/Q^\epsilon)\int_{\mathbb{R}}V\left(\frac{t}{T}\right)\sum_{1 \leq q \leq Q} \frac{g(q,x)}{q}\sideset{}{^\star}\sum_{a \, {\rm mod} \, q}   \\
		&\times \sum_{n=1}^{\infty} \lambda_{\pi}(n,r)e\left(\frac{an}{q}\right)e\left(\frac{nx}{qQ}\right)n^{it}V\left(\frac{n}{N}\right) \notag \\
		&\times \sum_{m=1}^{\infty}\lambda_f(m)m^{-it}e\left(\frac{-am}{q}\right)e\left(\frac{-mx}{qQ}\right)U\left(\frac{m}{N}\right)\mathrm{d}t \, \mathrm{d}x +O(k^{-2020}). \notag
	\end{align}
	\subsection{Sketch of  proof} \label{sketch} In this  subsection,  we will discuss rough  ideas to get non-trivial  cancellations in  $S_r(N)$ given in  \eqref{SN after dfi}.  For simplicity, we consider the generic case, i.e., $N=k^3$, $r=1$ and $q \sim Q=\sqrt{N/T}=k^{3/2}/T^{1/2}$.  Thus  $S_r(N) $ is roughly given by
	\begin{align*}
		\frac{1}{QT} \int_{T}^{2T} \sum_{q \sim Q}\frac{1}{q} \,  \sideset{}{^\star}{\sum}_{a \, {\rm mod} \, q} \, \sum_{n \sim N} \lambda_{\pi}(n,1) n^{it} e\left(\frac{an}{q}\right) \, \sum_{m \sim N} \, \lambda_{f}(m) \, m^{-it}  e\left(\frac{-am}{q}\right) \, \mathrm{d}t.
	\end{align*}
	Note that we have ignored the $x$-integral, as it does not contribute   in the generic case, and  we have also  supressed all the weight functions.  On estimating the above sum using Cauchy's inequality and the Rankin-Selberg bound, we get  $S_r(N) \ll N^{2+\epsilon}$. Our goal is to save $N$ plus a little more, say, $k^{\delta}$. In other words, we need to show $$S_r(N)\ll N^2/(Nk^\delta),$$
	for some $\delta>0$.  In the next step, we dualize  the sum over $n$ and $m$ (See Section   \ref{voronoi} for full details).

	Consider the sum over $n$
	$$\mathrm{S}_3=\sum_{n \sim N} \lambda_{\pi}(n,1) n^{it} e\left(\frac{an}{q}\right).$$
	On  applying the  $\mathrm{GL(3)}$ Voronoi summation formula to the above sum, we arrive at  (see Lemma \ref{n*})
	$$\mathrm{S}_3 \approx  \frac{N^{2/3}}{q}\sum_{n_2 \sim  Q^3T^3/N} \frac{\lambda_\pi(1, n_2)}{n_2^{1/3}}\,S( \bar a, \pm n_2; q)  \,\mathrm{I_3}(...),$$
	where $\mathrm{I_3}(...)$ is an integral transform in which we need to get square root cancellations, i.e., need to show  $\mathrm{I_3}(...) \ll 1/\sqrt{T}$. Next we apply the $\mathrm{GL(2)}$ Voronoi formula to the sum over $m$ and we get (see Lemma \ref{m*} for details)
	$$ \sum_{m \sim N} \, \lambda_{f}(m) \, m^{-it}  e\left(\frac{-am}{q}\right) \approx  \frac{N}{q} \sum_{ m \sim  Q^2k^2/N }\lambda_f(m)e\left(\frac{\bar{a}m}{q}\right)	\mathrm{I_2}(...),$$
	where $\mathrm{I_2}(...)$ is an integral transform 
	in which we need to get full cancellations, i.e., need to show $	\mathrm{I_2}(...)\ll 1/k$. Next we analyse the sum over $a$ which is given by 
	$$\mathfrak{C} = \sideset{}{^{*}}{\sum}_{a \, {\rm mod} \, q} S\left(\bar{a},n_2;q\right) \, e\left(\frac{\bar{a} m}{q}\right) \approx q e\left(-\frac{\bar{m} n_2}{q}\right).$$
	We observe that the above  sum becomes an additive character with respect to $n_2$ (which saves us extra ${q}$ when we apply the Poisson after Cauchy). Thus, we arrive at the following expression
	$$\frac{1}{QT}\frac{N}{Q^2T}\frac{N}{Q}	\sum_{q \sim Q}\sum_{n_2 \sim  T^{3/2}N^{1/2}} \lambda_{\pi}(1,n_2) \sum_{m \sim k^2/T } \lambda_{f}(m) e\left(-\frac{\bar{m} n_2}{q}\right)\mathfrak{J},$$
	where $\mathfrak{J}$ is an integral transform involving the $t$-integral, $\mathrm{I_2}(...)$ and  $\mathrm{I_3}(...)$.  We analyze it in Section \ref{estimates for int}. We observe that 
	$$\mathfrak{J} \ll T \frac{1}{\sqrt{T}}\frac{1}{\sqrt{T}}\frac{1}{k}.$$
	Note that a saving of $\sqrt{T}$ comes from the $t$-integral, another saving of $\sqrt{T}$ comes from the $\mathrm{GL(3)}$-integral and the saving of $k$ comes from the $\mathrm{GL(2)}$ integral. The factor $T$ reflects the length of the $t$-integral.  Thus, on plugging it in place of $\mathfrak{J}$, we see that 
	\begin{align*}
		S_r(N)&\ll \sum_{q \sim Q}\sum_{n_2 \sim  T^{3/2}N^{1/2}} |\lambda_{\pi}(1,n_2)|\bigg| \sum_{m \sim k^2/T } \lambda_{f}(m) e\left(-\frac{\bar{m} n_2}{q}\right)\mathfrak{J}\bigg|  \\
		&\ll QT^{3/2}N^{1/2}\frac{k^2}{T}\frac{1}{k} \ll Nk.
	\end{align*}
	Thus we now need to save $k^{1+\delta}$. Next we apply Cauchy's inequality to the sum over $n_2$ to get rid of the  $\mathrm{GL(3)}$ coefficients. Thus we arrive at (see Subsection \ref{cauchys})
	$$(T^{3/2}N^{1/2})^{1/2}\left(\sum_{n_2 \sim  T^{3/2}N^{1/2}} \Big \vert \sum_{q \sim Q} \sum_{m \sim k^2/T} \lambda_{f}(m) \, e\left( - \frac{\bar{m} n}{q}\right) \, \mathfrak{I}\Big \vert^{2}\right)^{1/2}.$$
	The end game strategy is to apply the Poisson to the sum over $n_2$ (see Subsection \ref{poisson}). Opening the absolute value square followed by the Poisson, we observe that  we save the whole length, i.e., $k^2Q/T$  in the zero-frequency ($n_2=0$ case) which suffices if   $k^2Q/T >k^2$ which implies that $T<k$. On the other hand,  in the non-zero frequencies ($n_2 \neq 0$ case), we save $$\frac{T^{3/2}N^{1/2}}{(Q^2T)^{1/2}}.$$ 
	Here the factor $Q^2T$ in the denominator reflects the size of the  conductor, which is given by
	$$\mathrm{arithmetic \  conductor } \times\mathrm{ analytic  \ conductor.}.$$ 
	Note that the 	arithmetic  conductor is of size $Q^2$ and the analytic conductor is of size $T$ (because  $\mathfrak{J}$ oscillates like $n_2^{iT}$ with respect to $n_2$). We also save $Q$ due to the presence of the additive character $ e( - {\bar{m} n}/{q})$. Thus the total saving in the non-zero frequencies turns out to be 
	$$\frac{T^{3/2}N^{1/2}}{(Q^2T)^{1/2}}\times Q=TN^{1/2},$$ 
	which suffices  if  $TN^{1/2} >k^2$ which boils down to $T >k^{1/2}$. Hence we have the restriction  $k^{1/2}<T<k$. In fact, the optimal  choice for $T$ turns out to be $k^{41/51}$. Thus we get  Theorem \ref{Main}.

	%
	%
	%
	\section{\bf Applications of  Voronoi   formulae}\label{voronoi}
	In this section,  we will analyse  the sum over $n$ and $m$ in \eqref{SN after dfi} using  Voronoi summation formulae.
	\subsection{$\textrm{GL(3)}$ Voronoi}
	Let's consider the sum over $n$ 
	\begin{align}\label{S3}
		\mathrm{S_3}:= \sum_{n=1}^{\infty} \lambda_{\pi}(n,r)e\left(\frac{an}{q}\right)e\left(\frac{nx}{qQ}\right)n^{it}V\left(\frac{n}{N}\right).
	\end{align}
 Recall that $N$ is of the form $N=2^\alpha$,  $\alpha \in [-1, \infty ) \cap \mathbb{Z}$,   such that $N \leq k^{3+\epsilon}/r^2$.    We analyze $\mathrm{S_3}$  using  the  $\mathrm{GL(3)}$ Voronoi summation formula (see  Lemma \ref{gl3voronoi}).  In the present set-up, we have $g(n)=e\left(nx/qQ\right)n^{it}V\left(n/N\right)$ and $X=N$. Thus, on applying Lemma \ref{gl3voronoi} to the above sum, we get
	\begin{align} \label{after gl3 voronoi}
		\mathrm{S_3}
		=q  \sum_{\pm} \sum_{n_{1}|qr} \sum_{n_{2}=1}^{\infty}  \frac{\lambda_{\pi}(n_1,n_2)}{n_{1} n_{2}} S\left(r \bar{a}, \pm n_{2}; qr/n_{1}\right) G_{\pm} \left(n_2^\star\right),
	\end{align} 
	where $n_2^\star :={n_{1}^2 n_{2}}/(q^3 r)$ and  $G_{\pm}(n_2^\star)$ is the integral transform defined in \eqref{gl3 integral transform}. Next we extract the oscillations of the integral transform $G_{\pm}(n_2^\star)$ using Lemma \ref{GL3oscilation}, which gives us  the following expression for $G_{\pm}(n_2^\star)$ in the range  $n_2^{\star} N \gg k^{\epsilon}$:  
	\begin{align*}
		G_{\pm} (n_2^\star)=n_2^\star \int_{0}^{\infty} g(z)\sum_{j=1}^{K_0} \frac{c_{j}({\pm}) e(3 (n_2^\star z)^{1/3} ) + d_{j}({\pm}) e(-3 (n_2^\star z)^{1/3} )}{\left( n_2^\star z\right)^{j/3}} \, \mathrm{d} z+O(k^{-2020}),
	\end{align*}
	where $K_0=[\frac{6060}{\epsilon}+5]+1$ with $[.]$ denoting  the greatest integer function. From now on, we will  continue our analysis with the terms corresponding to $j=1$, as the  other terms can be treated in a similar way and in fact,  give us  better estimates.
	Thus, on plugging the above expression corresponding to the term $j=1$ into \eqref{after gl3 voronoi}, we   arrive at 
	\begin{align*}
		\frac{N^{2/3+it}}{qr^{2/3}} \sum_{\pm}\sum_{n_1|qr}n_1^{1/3} &\sum_{n_2=1}^\infty \frac{\lambda_\pi(n_1,n_2)}{n_2^{1/3}}S(r \bar a, \pm n_2; qr/n_1)\mathrm{I_3}(n_1^2n_2, \, q,\, x), 
	\end{align*}
	where
	\begin{align}\label{gl3 integral after voronoi}
		\mathrm{I_3}(n_1^2n_2, \, q,\, x):=\int_0^\infty V(z)z^{it}e\left(\frac{Nxz}{qQ}\pm \frac{3(Nn_1^2n_2z)^{1/3}}{qr^{1/3}}\right)\mathrm{d}z.
	\end{align}
	On applying the change of variable  $z \mapsto z^3$ followed by    integration by parts (differentiating $3z^2V(z^3)z^{i3t}e(Nxz^3/qQ)$ and integrating $e(\pm 3(Nn_1^2n_2)^{1/3}z/qr^{1/3})$) $j$-times  to the above integral, we observe that  
	$$ \Big|\mathrm{I_3}(n_1^2n_2, \, q,\, x)\Big|\ll_j \left(1+T+\frac{N|x|}{qQ}\right)^j \left(\frac{qr^{1/3}}{(Nn_1^2n_2)^{1/3}}\right)^j,$$
	for any integer $j \geq 0$,  and it is negligibly small if 
	\begin{align} \label{n-sum bound}
		n_{1}^{2} n_{2} \gg  \,  \,  k^{\epsilon}\max \left\{\frac{q^3T^{3}r}{N},\,  T^{3/2} N^{1/2} r \right\}=: N_{0}.
	\end{align}
	Now it remains to analyse $G_{\pm}(n_2^\star)$ for $n_2^\star N \ll k^\epsilon$, which is given as 
	\begin{align}\label{gl3 integral transform 2 }
		G_{\pm}(n_2^\star) &=\frac{1}{2 \pi i} \int_{(\sigma)} (n_2^\star)^{-s} \, \gamma_{\pm}(s)\tilde{g}(-s)  \, \mathrm{d}s  \\
		& =\frac{1}{2 \pi } \int_{-\infty}^{\infty} (n_2^\star)^{-\sigma-i\tau} \, \gamma_{\pm}(\sigma+i\tau) \, \tilde{g}(-\sigma-i\tau) \, \mathrm{d}\tau. \notag
	\end{align}
	We will analyse  this   case   in Subsection \ref{estimates  for the error term}.
	We conclude this subsection by   summarising the above discussion in the following lemma.
	\begin{lemma} \label{n*}
		Let $\mathrm{{S_3}}$ be as  in \eqref{S3}. Then, for $n_2^\star N=n_1^2n_2N/(q^3r) \gg k^\epsilon$, we have  
		\begin{align}\label{S3 for generic}
			\mathrm{{S_3}}=\frac{N^{2/3+it}}{qr^{2/3}} \sum_{\pm}\sum_{n_1|qr}n_1^{1/3}\sum_{n_2 \ll N_0/n_1^2} &\frac{\lambda_\pi(n_1,n_2)}{n_2^{1/3}}\,S(r \bar a, \pm n_2; qr/n_1)  \,\mathrm{I_3}(n_1^2n_2, \, q,\, x)  \\
			&+{\mathrm{other \ lower \ order \ terms}}+O(k^{-2020}) \notag,
		\end{align}
		where $\mathrm{I_3}(n_1^2n_2, \, q,\, x) $ is an integral transform defined in  \eqref{gl3 integral after voronoi} and $N_0$ is as  defined in \eqref{n-sum bound}. For the non-generic case $n_2^\star N \ll k^\epsilon$, we have 
		\begin{align}  \label{S3 for nongeneric}
			\mathrm{S_3}
			=q  \sum_{\pm} \sum_{n_{1}|qr} \sum_{n_{2}=1}^{\infty}  \frac{\lambda_{\pi}(n_1,n_2)}{n_{1} n_{2}} S\left(r \bar{a}, \pm n_{2}; qr/n_{1}\right) G_{\pm} \left(n_2^\star\right),
		\end{align} 
		where $G_{\pm}(n_2^\star)$ is as defined in  \eqref{gl3 integral transform 2 }.
	\end{lemma}
	From now on, we will  proceed  with the main term of  \eqref{S3 for generic}.
	
	\subsection{$\textrm{GL(2)}$ Voronoi}\label{GL2 voronoi subsection}
	We  now consider the sum over $m$ in \eqref{SN after dfi}, which is given as 
	\begin{align}\label{S2}
		\mathrm{S_2}:=\sum_{m=1}^{\infty}\lambda_f(m)m^{-it}e\left(\frac{-am}{q}\right)e\left(\frac{-mx}{qQ}\right)U\left(\frac{m}{N}\right).
	\end{align}
	On applying  the $ \mathrm{GL(2)} $ Voronoi summation formula (see  Lemma \ref{gl2voronoi}) to the above sum with $g(m)=m^{-it}e(-mx/(qQ))U(m/N)$, we get
	\begin{align*}
		\mathrm{S_2}=\frac{2\pi i^kN^{1-it}}{q} \sum_{m=1}^{\infty}\lambda_f(m)e\left(\frac{\bar{a}m}{q}\right) 	\mathrm{I_2}(m,\,q,\,x),
	\end{align*}
 where 
	\begin{align}\label{gl2 transform}
		\mathrm{I_2}(m,\,q,\,x):=\int_0^{\infty}U(y) y^{-it}e\left(\frac{-Nxy}{qQ}\right)J_{k-1}\left(\frac{4\pi \sqrt{mNy}}{q}\right)\mathrm{d} y.
	\end{align}
	 We now analyse the above integral to determine the range of $m$.  We claim that $\mathrm{I_2}(m,\,q,\,x)$ is negligibly small unless
	\begin{align}\label{M_0}
		M:= \frac{q^2(k-1)^2k^{-\epsilon}}{N} \leq m \leq k^{\epsilon}\max\left(\frac{(k-1)^2q^2}{N},\,  T\right)=:M_0. 
	\end{align}
	In fact, in the range $m < M$,  we have $${4\pi \sqrt{mNy}}/{q} < 4\pi \sqrt{5/2}(k-1)^{1-\epsilon/2} \ll (k-1)^{1-\epsilon/2}.$$ Thus, by   Lemma \ref{bessel for x <k},   $\mathrm{I_2}(m,\,q,\,x)$  is negligibly small. 
	
	Next we consider   the range $m >  M_0$ and we claim that $\mathrm{I_2}(m,\,q,\,x)$ is also  negligibly small. We note that   ${4\pi \sqrt{mNy}}/{q} > (k-1)^{1+\epsilon/2}$. Thus  we    apply Langer's expansion (see Lemma \ref{langer series})  for   $J_{k-1}$. On applying Corollary \ref{Langer 2} with $x={4\pi \sqrt{mNy}}/{q}$, we see that $	\mathrm{I_2}(m,\,q,\,x)$, up to a negligible error term,  is given by  
	\begin{align*}
		\sum_{j=0}^{2019}\frac{1}{(k-1)^{j+1/2}}	\int_0^{\infty}U_j(y) y^{-it}e\left(\frac{-Nxy}{qQ}\right){e\left(\pm \frac{(k-1)(w-\tan^{-1}w)}{2\pi}\right)}\mathrm{d}y,
	\end{align*}
	where $U_j(y)=U(y)P_j\left(({w-\tan^{-1}w})^{-1}\right)w^{-1/2}$ with $$w=\left(\frac{x^2}{(k-1)^2}-1\right)^{1/2}=\left(\frac{16\pi^2mNy}{q^2(k-1)^2}-1\right)^{1/2},$$
	and $P_j$ is a polynomial of the  degree $j$ with  coefficients which are  bounded functions of $k$. Note that $w > ((k-1)^\epsilon-1)^{1/2}$. Thus $$w-\tan^{-1}w=w-\frac{\pi}{2}+\tan^{-1}\frac{1}{w} \asymp w,$$
	and $U_j^{(i)}(y) \ll_{i} k^{\epsilon i}$ for any integer $i \geq 0$.
	Next we  apply  integration by parts  $i$-times  to the   $y$-integral and  we get
	\begin{align*}
		|\mathrm{I_2}(m,\,q,\,x)| &\ll_i \left(k^\epsilon +T+\frac{N|x|}{qQ}\right)^i \left(\frac{1}{(k-1) \sqrt{mN}/(q(k-1))}\right)^i  \\
		& \ll \left(\frac{Tq}{\sqrt{M_0 N}}+\frac{N}{Q\sqrt{M_0N}}\right)^i \ll \left(\frac{k^{\epsilon} T}{k}+\frac{1}{k^{\epsilon}}\right)^i \ll \frac{1}{k^{\epsilon i}}.
	\end{align*}
	Upon taking $i$ sufficiently large, we get the claim. 
	We end this subsection by summarizing the above arguments in the following lemma.
	\begin{lemma} \label{m*}
		Let $\mathrm{S_2}$ be the sum over $m$ as given in \eqref{S2}. Then we have 
		\begin{align}\label{S2 final}
			\mathrm{S_2}=&\frac{2\pi i^kN^{1-it}}{q} \sum_{M\leq m \leq M_0}\lambda_f(m)e\left(\frac{\bar{a}m}{q}\right)	\mathrm{I_2}(m,\,q,\,x)+O(k^{-2020}),
		\end{align}
		where
		\begin{align*}
			\mathrm{I_2}(m,\,q,\,x)=\int_0^{\infty}U(y) y^{-it}e\left(\frac{-Nxy}{qQ}\right)J_{k-1}\left(\frac{4\pi \sqrt{mNy}}{q}\right)\mathrm{d} y, 
		\end{align*}
		and $M$ and $M_0$ are the ranges of $m$ defined in \eqref{M_0}. 
	\end{lemma}
		\section{\bf Cauchy and Poisson}
		After the applications of the Voronoi formulae and applying   Lemma \ref{n*} and Lemma \ref{m*} to  \eqref{SN after dfi}, we find that the expression in \eqref{SN after dfi}, up to an error term to be treated in Section \ref{estimates for the error term},  has been  essentially  reduced to  
		\begin{align} \label{s(N) after reduced}
			& \frac{N^{5/3}}{QTr^{2/3}}  \sum_{1 \leq q \leq Q}\frac{1}{q^3} \, \sideset{}{^\star}\sum_{a \, {\rm mod} \, q} \, \sum_{\pm}\sum_{n_1|qr}n_1^{1/3}  \\
			&  \times   \sum_{n_2\ll {N_0/n_1^2}} \frac{\lambda_\pi(n_1,n_2)}{n_2^{1/3}}S(r \bar a, \pm n_2; qr/n_1) \notag  \\
			&\times \sum_{M \ll m \ll M_0}\lambda_f(m)e\left(\frac{\bar{a}m}{q}\right) \mathrm{J}_{\pm}(m,\, n_1^2n_2,\, q) ,\notag
		\end{align}
		where 
		\begin{align}\label{J+-}
			\mathrm{J}_{\pm}(m,\, n_1^2n_2,\, q)&= \int_{\mathbb{R}} \int_{\mathbb{R}}W({x}/{Q^\epsilon}) \, g(q,\, x)\	\mathrm{I_2}(m,\,q,\,x)\, \mathrm{I_3}(n_1^2n_2, \, q,\, x)\, V\left(\frac{t}{T}\right)\,  \mathrm{d}t\, \mathrm{d}x.
		\end{align}
		In this section, we will  analyse  \eqref{s(N) after reduced}  using  Cauchy's  inequality and the Poisson summation formula.
		\subsection{Cauchy's  inequality}\label{cauchys}
		Splitting the sum over $q$ into dyadic blocks $q \sim C$, i.e.,  $C\leq q< 2C$, $C \ll Q$ and  writing $q=q_1q_2$ with $q_1|(n_1r)^\infty$, $(n_1r,\, q_2)=1$, we see that  the expression in \eqref{s(N) after reduced}  is dominated by
		\begin{align}\label{Q-dyadic block}
			\sup_{C \ll Q}\frac{N^{5/3}\log Q}{QTr^{2/3}C^3} \, \sum_{\pm}\sum_{\frac{n_1}{(n_1,r)}\ll C}n_1^{1/3}\sum_{\frac{n_1}{(n_1,r)}|q_1|(n_1r)^\infty}\sum_{n_2\ll {N_0}/{n_1^2}} \frac{|\lambda_\pi(n_1,n_2)|}{n_2^{1/3}} \\ 
			\times \,  \Big|\sum_{q_2\sim C/q_1}\sum_{M \ll m \ll M_0}\lambda_f(m)\mathcal{C_{\pm}}(q, n_2,m)\mathrm{J}_{\pm}(m,n_1^2n_2,q)\Big|\notag ,
		\end{align}
		where the character sum $\mathcal{C_{\pm}}(q,n_2,m)=	\mathcal{C_{\pm}}(. . .)$ is defined as 
		\begin{align*}
			\mathcal{C_{\pm}}(. . .):=\sideset{}{^\star}\sum_{a \, {\rm mod} \, q}S(r \bar a, \pm n_2; qr/n_1)e\left(\frac{\bar{a}m}{q}\right)=\sum_{d|q}d\mu\left(\frac{q}{d}\right)\sideset{}{^ \star}\sum_{\substack{\alpha \, {\rm mod} \, qr/n_1 \\ n_1\alpha\equiv-m \, {\rm mod} \, d}}e\left(\pm\frac{\bar{\alpha}n_2}{qr/n_1}\right).
		\end{align*}  
		Next we  analyse the expression  inside $| \ |$.  We first  split the sum over $m$ into dyadic blocks $m \sim M_1$, $M\ll M_1\ll M_0$ and then   apply  Cauchy's inequality to the sum  over $n_2$ in \eqref{Q-dyadic block} to arrive at
		\begin{align}\label{S(N)after cauchy}
			S_r(N) \ll \mathop{\sup_{ \substack{M\ll M_1\ll M_0 \\ C \ll Q}}}\frac{N^{5/3}(QM_0)^{\epsilon}}{QTr^{2/3}C^3}\sum_{\pm}\sum_{\frac{n_1}{(n_1,r)}\ll C}n_1^{1/3}\Theta^{1/2}\sum_{\frac{n_1}{(n_1,r)}|q_1|(n_1r)^\infty}\sqrt{\Omega_{\pm}},
		\end{align}
		where 
		\begin{align}\label{theta}
			\Theta=\sum_{n_2\ll N_0/n_1^2} \frac{|\lambda_\pi(n_1,n_2)|^2}{n_2^{2/3}},
		\end{align} 
		and 
		\begin{align}\label{omega}
			\Omega_{\pm}=\sum_{n_2\ll N_0/n_1^2}\Big|\sum_{q_2\backsim C/q_1}\sum_{m \backsim M_1}\lambda_f(m)\mathcal{C}_{\pm}(q,n_2,m)\mathrm{J}_{\pm}(m,n_1^2n_2,q)\Big|^2,
		\end{align}
		with  
		\begin{align} \label{N0}
			\frac{(k-1)^2C^2}{N}k^{-\epsilon}&=M\ll M_1 \ll M_0= k^{\epsilon}\max\left(\frac{(k-1)^2C^2}{N},\,  T\right),   \\ 
			N_0&=k^{\epsilon}\max \left\{\frac{\left(CT\right)^{3}r}{N},\,  T^{3/2} N^{1/2}r \right\}.\notag
		\end{align}

		\subsection{ Poisson summation }  \label{poisson}Next we  apply the Poisson summation formula to the sum  over $n_2$ with the  modulus $\mathfrak{q}:=q_1q_2q_2^\prime r/n_1$ in \eqref{omega}. To this end, we first split the sum over $n_2$ into dyadic blocks $n_2 \sim \tilde{N}/n_1^2$, $\tilde{N} \ll N_0$. Then opening the absolute value square in \eqref{omega}, we arrive at 
		\begin{align*}
			\Omega_{\pm}=&\mathop{\sum \sum}_{q_2, \, q_2^{\prime}\sim C/q_1}\mathop{\sum \sum}_{m,\, m^{\prime} \sim M_1}\lambda_f(m) \overline{\lambda_f(m^\prime)}\,  \Delta_{\pm},
		\end{align*}
		where 
		$$ \Delta_{\pm}= \sum_{\tilde{N}}\sum_{n_2 \in \mathbb{Z}}\phi\left(\frac{n_1^2n_2}{\tilde{N}}\right) \mathcal{C}_{\pm}(q,n_2,m)  \overline{\mathcal{C}_{\pm }(q^\prime,n_2,m^\prime)}\mathrm{J}_{\pm}(m,n_1^2n_2,q)\overline{\mathrm{J}_{\pm}(m^{\prime},n_1^2n_2,q^{\prime})},$$
		$q^{\prime}=q_1q_2^{\prime}$ and $\phi(w)$  is a non-negative smooth  function supported on  $[2/3,\, 3]$ with $\phi(w)=1$ for $w \in [1, \, 2]$ and $\phi^{(j)}(w) \ll_j 1$. Now 
		applying the change of variable
		$$n_2 \rightarrow n_2\mathfrak{q}+\beta, \  \  \   \beta \ \mathrm{mod}\  \mathfrak{q},$$
		we get the following expression for $\Delta_{\pm}$: 
		\begin{align*}
			\Delta_{\pm}=&\sum_{\tilde{N}} \sum_{\beta \; {\rm mod} \; \mathfrak{q} } \mathcal{C}_{\pm}(q,\beta,m)  \overline{\mathcal{C}_{\pm }(q^\prime,\beta,m^\prime)}  \\
			& \times \sum_{n_2 \in \mathbb{Z}} 	 \phi\left(\frac{n_2\mathfrak{q}+\beta}{\tilde{N}/n_1^2}\right) \mathrm{J}_{\pm}(m,n_1^2(n_2\mathfrak{q}+\beta), q)\,  \overline{\mathrm{J}_{\pm}(m^{\prime},n_1^2(n_2\mathfrak{q}+\beta),q^{\prime})}.
		\end{align*}
		On applying the Poisson summation formula to the sum over $n_2$, we  see that 
		\begin{align}\label{final omega}
			\Omega_{\pm}=&\sum_{\tilde{N}}\frac{\tilde{N}}{n_1^2}\mathop{\sum \sum}_{q_2,\, q_2^{\prime}\sim C/q_1}\mathop{\sum \sum}_{m,\, m^{\prime} \sim M_1}\lambda_f(m)\overline{ \lambda_f(m^\prime)} \sum_{n_2 \in \mathbb{Z}}\mathfrak{C}_{\pm}\mathcal{J}_{\pm},
		\end{align}
		where 
		\begin{align}\label{character sum}
			\mathfrak{C}_{\pm}&= \frac{1}{\mathfrak{q}}\sum_{\beta \; {\rm mod} \; \mathfrak{q} } \mathcal{C}_{\pm}(q,\beta,m) \, \overline{\mathcal{C}_{\pm }(q^\prime,\beta,m^\prime)} \, e\left(\frac{n_2\beta}{\mathfrak{q}}\right) \\
			&=\mathop{\sum \sum}_{\substack{d|q \\ d^{\prime}|q^\prime}}dd^\prime\mu\left(\frac{q}{d}\right)\mu\left(\frac{q^\prime}{d^\prime}\right)\mathop{\sideset{}{^ \star}\sum_{\substack{\alpha \, {\rm mod} \, qr/n_1 \\ n_1\alpha\equiv-m \, {\rm mod} \, d}} \  \sideset{}{^ \star}\sum_{\substack{\alpha^\prime \, {\rm mod} \, q^\prime r/n_1 \\ n_1\alpha^\prime \equiv-m^\prime \, {\rm mod} \, d^\prime}}}_{\pm \bar{\alpha}q_2^\prime \mp\bar{\alpha}^\prime q_2\equiv -n_2  \, {\rm mod} \, \mathfrak{q}}1 \notag,
		\end{align}
		and 
		\begin{align}
			\mathcal{J}_{\pm}&= \int_{\mathbb{R}} \phi(w) \, \mathrm{J}_{\pm}(m,\tilde{N}w, q)\,  \overline{\mathrm{J}_{\pm}(m^{\prime},\tilde{N}w,q^{\prime})}\,  e\left(-\frac{n_2\tilde{N} w}{q_1q_2q_2^\prime r n_1}\right) \, \mathrm{d}w. 
		\end{align}
		On estimating the sum over $\tilde{N}$,   we  get
		\begin{align}  \label{Omega}
			\Omega_{\pm} \ll  k^{\epsilon} \sup_{\tilde{N} \ll N_0}\frac{\tilde{N}}{n_1^2}\mathop{\sum\,  \sum}_{q_2,\, q_2^{\prime}\sim C/q_1}\mathop{\sum\,  \sum}_{m,\, m^{\prime} \sim M_1} |\lambda_f(m)||\lambda_{f}(m^\prime)| \sum_{n_2 \in \mathbb{Z}}|\mathfrak{C}_{\pm}||\mathcal{J}_{\pm}|.
		\end{align}
	
		\section{\bf Estimates for the  integral transform}\label{estimates for int}
		In this section, we will analyse the  integral transform 
		\begin{align}\label{ integral}
			\mathcal{J}_{\pm}&= \int_{\mathbb{R}} \phi(w) \, \mathrm{J}_{\pm}(m,\tilde{N}w, q)\,  \overline{ \mathrm{J}_{\pm}(m^{\prime},\tilde{N}w,q^{\prime})}\,  e\left(-\frac{n_2\tilde{N} w}{q_1q_2q_2^\prime r n_1}\right) \, \mathrm{d}w,
		\end{align}
		where (see \eqref{J+-}) 
		\begin{align}\label{integral before cauchy}
			\mathrm{J}_{\pm}(m,\, \tilde{N}w,\, q)=& \int_{\mathbb{R}} \int_{\mathbb{R}}W(x/Q^\epsilon) \, g(q,\, x)\	\mathrm{I_2}(m,\,q,\,x)\, \mathrm{I_3}(\tilde{N}w, \, q,\, x)\, V\left(\frac{t}{T}\right)\,  \mathrm{d}t\, \mathrm{d}x \\
			=&\int_{\mathbb{R}}W(x/Q^\epsilon)g(q,x) \int_{\mathbb{R}}V\left(\frac{t}{T}\right)\int_0^{\infty}U(y)y^{-it}\int_0^{\infty}V(z)z^{it} \notag \\
			& \times e\left(\frac{Nx(z-y)}{qQ}\pm \frac{3(N\tilde{N}wz)^{1/3}}{qr^{1/3}}\right)J_{k-1}\left(\frac{4\pi \sqrt{mNy}}{q}\right)\, \mathrm{d}z\, \mathrm{d}y\,  \mathrm{d}t\, \mathrm{d}x \notag.
		\end{align}
		and $\mathrm{J}_{\pm}(m^{\prime},\tilde{N}w,q^{\prime})$ is similarly defined. We first analyse $\mathrm{J}_{\pm}(m,\, \tilde{N}w,\, q)$. 
		%
	\begin{lemma}\label{restriction on u }
			Let $\mathrm{J}_{\pm}(m,\, \tilde{N}w,\, q)$ be  as above. Then  we have 
			\begin{align}\label{reduced integral}
				\mathrm{J}_{\pm}(m,\, \tilde{N}w,\, q)=
				\int_{\mathbb{R}}  V\left(\frac{t}{T}\right)   \int_{u \ll \frac{k^\epsilon C}{QT}} \,  \mathrm{I}_u\, 	\mathrm{I}_{\pm}(m,\tilde{N}w,q)\,  \mathrm{d}u \, \mathrm{d} t +O(k^{-2020}),
			\end{align}
			where $\mathrm{I}_u$  and $\mathrm{I}_{\pm}(m,\tilde{N}w,q)$  are the  integrals    defined in \eqref{x integral}  and \eqref{y-integral} respectively,  with the weight function $U_{u,t}$ satisfying 	$U_{u,t}^{(j)}(y) \ll_j k^{\epsilon j}$ for $j \geq 0$.

		\end{lemma}
		\begin{proof} 
			We  consider two cases. 
			\paragraph{Case 1.}  $q \sim C \ll Q^{1-\epsilon}$.

			Consider the integral over $x$ in \eqref{integral before cauchy} which is given by 
			\begin{align*}
				\mathrm{I}_{z-y}:&=\int_{\mathbb{R}}W(x/Q^{\epsilon})g(q,x)e\left(\frac{Nx(z-y)}{qQ}\right)\, \mathrm{d}x \\
				&=Q^{\epsilon} \int_{\mathbb{R}}W(x)g(q,xQ^\epsilon)e\left(\frac{NxQ^{\epsilon}(z-y)}{qQ}\right)\, \mathrm{d}x.
			\end{align*}
			We now split the above integral as 
			\begin{align*}
				\int_{\mathbb{R}}... \, \mathrm{d}x=	\int_{-Q^{-2\epsilon}}^{Q^{-2\epsilon}}... \, \mathrm{d}x+	\int_{D}... \, \mathrm{d}x,
			\end{align*}
			where $D=[-2, \,2] \backslash [-Q^{-2\epsilon}, \, Q^{-2\epsilon}]$. Note that,  for $x \in [-Q^{-2\epsilon}, \, Q^{-2\epsilon}] $, we have 
			$$ g(q,xQ^{\epsilon})=1+h(q,xQ^\epsilon)=  1+O \left(\frac{Q}{q} \left(\frac{q}{Q}+|x|Q^\epsilon\right)^{B}\right) =1+O(Q^{-2020}).$$
			Thus, in this range, we can replace $g(q,xQ^\epsilon)$ by $1$ at the cost of a negligible error term.  Then  by repeated integration by parts we see that the integral is negligibly small unless 
			\begin{align}\label{restriction on z and y}
				|z-y| \ll k^{\epsilon}C/(QT).
			\end{align}
			Now we consider the complementary range, i.e.,  $x \in D$. Note that, using the second  property (see \eqref{g properties}) of $g(q,x)$,  we have 
			$$ x^j \frac{\partial ^j}{\partial x^j}g(q,x) \ll \log Q \min \left\lbrace \frac{Q}{q}, \frac{1}{|x|}\right\rbrace \ll Q^{2\epsilon}.$$
			Thus, on using integration by parts repeatedly, we see that  the integral is negligibly small unless \eqref{restriction on z and y} holds true. 
			\paragraph{Case 2.}$q \sim C \gg Q^{1-\epsilon}$. 
			
			In this case, we consider the $t$-integral in \eqref{integral before cauchy} which is given by 
			$$\int_{\mathbb{R}}V\left(\frac{t}{T}\right) \left(\frac{z}{y}\right)^{it}\, \mathrm{d}t.$$
			Now applying  the change of variable $t \rightarrow tT$ followed by integration by parts repeatedly, we conclude that the $t$-integral is negligibly small unless  $$|z-y| \ll k^{\epsilon}/T\ll  k^{\epsilon}C/(QT).$$
			Next writing   $z-y=u$ with $u \ll k^\epsilon C/(QT)$ in \eqref{integral before cauchy}, we see that  
			\begin{align}
				\mathrm{J}_{\pm}(m,\, \tilde{N}w,\, q)=
				\int_{\mathbb{R}}  V\left(\frac{t}{T}\right)   \int_{u \ll \frac{k^\epsilon C}{QT}} \,  \mathrm{I}_u\, 	\mathrm{I}_{\pm}(m,\tilde{N}w,q)\,  \mathrm{d}u \, \mathrm{d} t +O(k^{-2020}),
			\end{align}
			where  
			\begin{align}\label{x integral}
				\mathrm{I}_{u}=\int_{\mathbb{R}}W(x/Q^{\epsilon})g(q,x)e\left(\frac{Nxu}{qQ}\right)\, \mathrm{d}x,
			\end{align}
			and 
			\begin{align}\label{y-integral}
				\mathrm{I}_{\pm}(m,\tilde{N}w,q)=\int_0^{\infty}\, U_{u,t}(y)
				e\left(\pm \frac{3(N\tilde{N}w(y+u))^{1/3}}{qr^{1/3}}\right)J_{k-1}\left(\frac{4\pi \sqrt{mNy}}{q}\right)\mathrm{d}y,
			\end{align}
			with  $U_{u,t}(y)=U(y)V(y+u)(1+u/y)^{it}$.
			Note that $$ \frac{\partial ^j}{\partial y^j}\left(1+\frac{u}{y}\right)^{it}= \frac{\partial ^j}{\partial y^j}\exp\left({it\log\left(1+\frac{u}{y}\right)}\right)\ll_j k^{\epsilon j},  \ \ \  j \geq 0.$$ Thus
			$U_{u,t}^{(j)}(y) \ll_j k^{\epsilon j}$ for $j \geq 0$. Hence the lemma follows. 

		\end{proof}
		The analysis for  $\mathrm{J}_{\pm}(m^{\prime},\tilde{N}w,q^{\prime})$  is exactly same. 
	Thus on plugging the expression of $ \mathrm{J}_{\pm}(m,\tilde{N}w, q)$  from \eqref{reduced integral} and a corresponding expression of  $\mathrm{J}_{\pm}(m^{\prime},\tilde{N}w,q^{\prime})$   into \eqref{  integral}, we see that 
		\begin{align}\label{main integral after u}
			\mathcal{J}_{\pm}=	\int_{\mathbb{R}}  \int_{\mathbb{R}}  V\left(\frac{t}{T}\right)  V\left(\frac{t^\prime}{T}\right) \int_{u \ll \frac{k^\epsilon C}{QT}} \,   \int_{u^\prime \ll \frac{k^\epsilon C}{QT}} \,\mathrm{I}_u\,  \overline{ \mathrm{I}}_{u^\prime}  \,	\mathfrak{J}_{\pm} \,  \mathrm{d}u^\prime  \,  \mathrm{d}u \,  \mathrm{d} t^\prime   \,  \mathrm{d} t+O(k^{-2020}),
		\end{align}
		where 
		\begin{align}\label{J}
			\mathfrak{J}_{\pm}:=\int_{\mathbb{R}}\phi(w)\, \mathrm{I}_{\pm}(m,\tilde{N}w,q)\, \overline{\mathrm{I}_{\pm}(m^\prime,\tilde{N}w,q^\prime)}\, e\left(-\frac{ n_2 \tilde{N}w}{q_2q_2^\prime q_1r n_1}\right) \, \mathrm{d}w,
		\end{align}
		which we will analyse now. 	We  have  the following proposition.
		\begin{proposition}\label{bound for integral} 
			Let $\mathfrak{J}_{\pm}$  be  as above.  Then $\mathfrak{J}_{\pm}$ is negligibly small unless 
			\begin{align} \label{N2}
				n_2 \ll k^{\epsilon}\frac{CN^{1/3}r^{2/3}n_1}{q_1\tilde{N}^{2/3}}:=N_2,
			\end{align}
			in which case  we have
			\begin{align}\label{integral bound for all case}
				\mathfrak{J}_{\pm} \ll \frac{k^\epsilon C^2}{M_1N}. 
			\end{align}
			Furthermore, if $q \sim C \gg k^{1+\epsilon}$ and $n_2\neq 0$, then we have  
			\begin{align}\label{y integral  bound for large c}
				\mathfrak{J}_{\pm}  \ll  \frac{Cr^{1/3}k^{2/3}}{k^2(N\tilde{N})^{1/3}}  .
			\end{align}
		\end{proposition}

		Before proving the  proposition, we will  analyze  $\mathrm{I}_{\pm}(m,\tilde{N}w,q)$ and  $\mathrm{I}_{\pm}(m^\prime,\tilde{N}w,q^\prime)$. We have the following lemma. 
		\begin{lemma}\label{y-integral bound}
			Let $\mathrm{I}_{\pm}(m,\tilde{N}w,q)$ be  the integral transform defined  in \eqref{y-integral}. Let $\mathfrak{b}=4\pi \sqrt{mN}/q$  and  $\mathfrak{a}=\mathfrak{a}(q,r):={3(N\tilde{N})^{1/3}}/(qr^{1/3}) \gg k^\epsilon$.  Then  $\mathrm{I}_{\pm}(m,\tilde{N}w,q)$ is negligibly small unless $\mathfrak{a} \leq k^{\epsilon} \mathfrak{b}$ . In the   case  when $\mathfrak{a} \leq k^{-\epsilon} \mathfrak{b}$, we have 
			\begin{align*}
				\mathrm{I}_{\pm}(m,\tilde{N}w,q)  \ll {k^\epsilon}/{\mathfrak{b}}.
			\end{align*}
			Furthermore, if $q\sim C \gg k^{1+\epsilon}$, then $\mathfrak{b} \asymp k$ and  we have 
			\begin{align} \label{y-integral expansion}
				\mathrm{I}_{\pm}(m,\tilde{N}w,q)=&\frac{e\left(f(\tau_0) \right) }{\sqrt{f^{\prime \prime}(\tau_0)}}\frac{c_3\mathfrak{a}^{9/2}w^{3/2}}{\mathfrak{b}^5\tau_0^5 \sqrt{1-\tau_0^2}} U_{u,t}\left(\left(\frac{4\pi \mathfrak{a}w^{1/3}}{3\mathfrak{b}  \tau_0}\right)^6\right) \\
				&+\mathrm{lower \  order\ terms} 
				+O\left(k^{-2020}\right), \notag
			\end{align}
			where  $\tau_0$ is the stationary point of the phase function  $$f(\tau)=\frac{(k-1)\sin^{-1}\tau}{2 \pi}+\frac{16\pi ^2\mathfrak{a}^3w}{27\mathfrak{b}^2\tau^2},$$ which is given by \eqref{stat point} and $c_3=c_2e(1/8)=3\sqrt{2}(4\pi/3)^{5}e(1/4)$. In the  remaining case, i.e.,  $k^{-\epsilon} \mathfrak{b} \leq \mathfrak{a} \leq k^{\epsilon} \mathfrak{b}$, $\mathrm{I}_{\pm}(m,\tilde{N}w,q) $ essentially looks like
			\begin{align*} 
				\frac{c_2\mathfrak{a}^{9/2}w^{3/2}}{\mathfrak{b}^5}	\int_{b_1/2}^{1} \frac{1}{ \tau^5\sqrt{1-\tau^2} } U_{u,t}\left(\left(\frac{4 \pi \mathfrak{a}w^{1/3}}{3\mathfrak{b}  \tau}\right)^6\right)e\left( f(\tau)\right)	 \mathrm{d}\tau,
			\end{align*}
			where $b_1:={4\pi(2/3)^{1/3}\mathfrak{a}}/{(3(2.5)^{1/6}\mathfrak{b})}$.
		\end{lemma}
		
		\begin{proof}
			Let's recall from \eqref{y-integral} that 
			\begin{align}\label{y integral 2}
				\mathrm{I}_{\pm}(m,\tilde{N}w,q)=\int_{1/2}^{5/2}\, U_{u,t}(y)\, 
				e\left(\pm \mathfrak{a}w^{1/3}(y+u)^{1/3}\right)J_{k-1}\left( \mathfrak{b} \sqrt{y}\right)\, \mathrm{d}y. 
			\end{align}
			Consider   the term $e(\pm \mathfrak{a}w^{1/3}(y+u)^{1/3})$. It  can be written  as
			$$e(\pm \mathfrak{a}w^{1/3}(y+u)^{1/3})=e(\pm \mathfrak{a}w^{1/3}y^{1/3})\, e(\pm \mathfrak{a}w^{1/3}y^{1/3}((1+u/y)^{1/3}-1)).$$ 
			Note that  
			$$ \frac{\partial ^j}{\partial y^j}e(\pm \mathfrak{a}w^{1/3}y^{1/3}((1+u/y)^{1/3}-1)) \ll_j k^{\epsilon j}, \  \ \ j \geq 0.$$ 
			 This is obvious for $j=0$.  We will verify it for $j=1$ (for other $j$, a similar calculation will follow).   Let $h(y,w):= \pm \mathfrak{a}w^{1/3}y^{1/3}((1+u/y)^{1/3}-1)$. Thus for $j=1$ we have 
			\begin{align*}
				\frac{\partial }{\partial y}e(h(y,w))&=e(h(y,w))(\pm \mathfrak{a})w^{1/3}\left(\frac{ (1+u/y)^{1/3}-1}{3y^{2/3}}-\frac{u}{3y^{5/3}(1+u/y)^{2/3}}\right).
			\end{align*}
			Thus, using $y,\, w \asymp 1$ and $(1+u/y)^{1/3}-1 \ll |u|$, we see that 
			$$\frac{\partial }{\partial y}e(h(y,w))\ll \mathfrak{a}|u| \ll  \frac{(N\tilde{N})^{1/3}}{Cr^{1/3}}\frac{Ck^\epsilon}{Q{T}}\ll  \frac{(N{N_0})^{1/3}}{Qr^{1/3}}\frac{Qk^\epsilon}{Q{T}}\ll k^\epsilon ,$$
			where we used \eqref{N0}  to estimate $N_0$. 
			Hence we  can  insert $e(h(y,w))$ into the weight function $U_{u,t}(y)$.  
			Thus we arrive at the following expression:
			\begin{align}\label{y integral without u}
				\mathrm{I}_{\pm}:=\mathrm{I}_{\pm}(m,\tilde{N}w,q)=\int_{1/2}^{5/2}\, U_{u,t}(y)\, 
				e\left(\pm \mathfrak{a}w^{1/3}y^{1/3}\right)J_{k-1}\left( \mathfrak{b} \sqrt{y}\right)\, \mathrm{d}y.
			\end{align}
			Notice  the slight abuse of notation: the weight function $U_{u,t}$ in the above expression  is different from the one in \eqref{y integral 2}. To analyze \eqref{y integral without u} further, we use an integral representation  of the Bessel function $J_{k-1}$.  Thus, on applying  \eqref{Bessel defi}  to  the Bessel function $J_{k-1}$, we see that 
			\begin{align*}
				\mathrm{I}_{\pm}=\frac{1}{2\pi}\int_{-\pi}^{\pi}e ^{i(k-1)\tau} \int_{1/2}^{5/2}U_{u,t}(y)e\left( \pm \mathfrak{a}w^{1/3}y^{1/3}-{\mathfrak{b}\sqrt{y}\sin \tau}/{2\pi}\right)\mathrm{d}y\, \mathrm{d}\tau.
			\end{align*}
			We now split the  $\tau$-integral as follows:
			\begin{align*}
				\int_{-\pi}^{\pi}...\,\mathrm{d}\tau=\int_{0}^{\pi/2}...\, \mathrm{d}\tau+\int_{\pi/2}^{\pi}...\, \mathrm{d}\tau+\int_{-\pi/2}^{0}...\, \mathrm{d}\tau+\int_{-\pi}^{-\pi/2}...\,\mathrm{d}\tau.
			\end{align*}
			Let $\mathrm{I}_{\pm}^{(i)}$ denote the  $i$-th integral in the right hand side  of the above expression  for $i=1,\, 2,\, 3$ and $4$.  Let's first consider   $\mathrm{I}_{\pm}^{(1)}$ which is defined as follows:
			\begin{align}\label{first range of tau int}
				\mathrm{I}_{\pm}^{(1)}=\frac{1}{2\pi}\int_{0}^{\pi/2}e ^{i(k-1)\tau} \int_{1/2}^{5/2}U_{u,t}(y)e\left( \pm \mathfrak{a}w^{1/3}y^{1/3}-{\mathfrak{b}\sqrt{y}\sin \tau}/{2\pi}\right)\mathrm{d}y\, \mathrm{d}\tau.
			\end{align}
			Next we apply stationary phase analysis to the $y$-integral. By  the change of variable $y \rightarrow y^3$, we arrive at the following expression of the $y$-integral:
			\begin{align*}
				\int_{\sqrt[3]{1/2}}^{\sqrt[3]{5/2}}3y^2U_{u,t}(y^3)e\left( \pm \mathfrak{a}w^{1/3}y-{\mathfrak{b}y^{3/2}\sin \tau}/{2\pi}\right)\textrm{d}y.
			\end{align*}
			Note that if we have negative sign with $\mathfrak{a}$, then the above integral is negligibly small by Lemma \ref{derivative bound}. Thus we  proceed with the $y$-integral of $\mathrm{I}_{+}^{(1)}$, which is given by 
			\begin{align*}
				\int_{\sqrt[3]{1/2}}^{\sqrt[3]{5/2}}3y^2U_{u,t}(y^3)e\left(  \mathfrak{a}w^{1/3}y-{\mathfrak{b}y^{3/2}\sin \tau}/{2\pi}\right)\textrm{d}y.
			\end{align*}
			Here the phase function is given by $f_1(y)=\mathfrak{a}w^{1/3}y-{\mathfrak{b}y^{3/2}\sin \tau}/{2\pi}$. On computing the first order derivative, we see that the  stationary point occurs at $y_0=\left(\frac{4 \pi \mathfrak{a}w^{1/3}}{3\mathfrak{b}\sin \tau}\right)^2$.  
			Note that $$\sqrt[3]{1/2}\leq y_0\leq\sqrt[3]{5/2} \iff \frac{4\pi}{3}\frac{\mathfrak{a}w^{1/3}}{\mathfrak{b}(2.5)^{1/6}} \leq \sin \tau \leq  \frac{4\pi}{3}\frac{\mathfrak{a}w^{1/3}}{\mathfrak{b}(0.5)^{1/6}}.$$
			Let $b_1:=\frac{4\pi}{3}\frac{\mathfrak{a}(2/3)^{1/3}}{\mathfrak{b}(2.5)^{1/6}}$
			and $b_2:= \frac{4\pi}{3}\frac{3^{1/3}\mathfrak{a}}{\mathfrak{b}(0.5)^{1/6}}$.
			We consider three cases. 
			\paragraph{Case 1} $\mathfrak{a} \geq k^\epsilon \mathfrak{b}$.
			In this case we have $b_1\geq 2$. Thus there is no stationary point in the range $[(1/2)^{1/3}, \, (5/2)^{1/3}]$. Moreover, $$f_1^{\prime}(y)=\mathfrak{a}w^{1/3}-3\mathfrak{b}\sqrt{y}\sin \tau /(4\pi) \gg  \mathfrak{b},  \ f_1^{(j)}(y) \ll  \mathfrak{b}, \ j\geq 2.$$ 
			Hence, by Lemma \ref{derivative bound}, the integral is negligibly small. This  proves the first part of the lemma.  
			\paragraph{Case 2} $\mathfrak{a} \leq k^{-\epsilon} \mathfrak{b}.$
			In this case  we have $0<b_1/2 <2b_2 \ll k^{-\epsilon} <1.$   we  now  split   the $\tau$-integral in \eqref{first range of tau int}  as follows: 
			$$\int_{0}^{\pi/2} ... \, \mathrm{d}\tau=\int_{0}^{\sin^{-1}(b_1/2)} ... \, \mathrm{d}\tau+\int_{\sin^{-1}(b_1/2)}^{\sin^{-1}2b_2} ... \, \mathrm{d}\tau + \int_{\sin^{-1}2b_2}^{\pi/2} ... \, \mathrm{d}\tau.$$
			Note that the first and the third integrals of the right side of the above expression  are negligibly small due to  absence of the stationary point.  Hence it boils down to analyse   the second integral which is given by
			\begin{align} \label{second integral}
				\int_{\sin^{-1}(b_1/2)}^{\sin^{-1}2b_2}e ^{i(k-1)\tau} 	\int_{\sqrt[3]{1/2}}^{\sqrt[3]{5/2}}3y^2U_{u,t}(y^3)e\left(  \mathfrak{a}w^{1/3}y-{\mathfrak{b}y^{3/2}\sin \tau}/{2\pi}\right)\textrm{d}y\, \mathrm{d}\tau.
			\end{align}
			On applying the stationary phase analysis (see Lemma \ref{stationaryphase}) to the  $y$-integral, we see that, it  is given by 
			$$\frac{c_1y_0^2U_{u,t}(y_0^3)e(f_1(y_0))}{\sqrt{|f_1^{\prime \prime}(y_0)|}}+{\mathrm{lower \ order \ terms}}+O(k^{-2020}),$$
			where $c_1=3 e(1/8)$, $y_0=\left(\frac{4 \pi \mathfrak{a}w^{1/3}}{3\mathfrak{b}\sin \tau}\right)^2$ and  $f_1(y)=\mathfrak{a}w^{1/3}y-{\mathfrak{b}y^{3/2}\sin \tau}/{2\pi}$. 
			We will proceed with the main term, as the  other terms can be analysed similarly and in fact, give better bounds.  Hence, on plugging the values of $y_0$, $f_1(y_0)$ and $f_1^{\prime \prime}(y_0)$, we essentially get  the following expression for the $y$-integral:
			\begin{align}
				\frac{c_2\mathfrak{a}^{9/2}w^{3/2}}{\mathfrak{b}^5\sin^5 \tau} U_{u,t}\left(\left(\frac{4 \pi \mathfrak{a}w^{1/3}}{3\mathfrak{b} \sin \tau}\right)^6\right)e\left( \frac{16\pi^2 \mathfrak{a}^3w}{27\mathfrak{b}^2\sin^2\tau}\right),
			\end{align}
			where $c_2=c_1\sqrt{2}(4\pi/3)^5$. On plugging the above expression in place of the $y$-integral into \eqref{second integral}, we arrive at 
			\begin{align*}
				\frac{c_2\mathfrak{a}^{9/2}w^{3/2}}{\mathfrak{b}^5}	\int_{\sin^{-1}(b_1/2)}^{\sin^{-1}2b_2} \frac{1}{\sin ^5 \tau} U_{u,t}\left(\left(\frac{4 \pi \mathfrak{a}w^{1/3}}{3\mathfrak{b} \sin \tau}\right)^6\right)e\left( \frac{(k-1)\tau}{2\pi}+\frac{16\pi^2 \mathfrak{a}^3w}{27\mathfrak{b}^2\sin^2\tau}\right)	 \mathrm{d}\tau.
			\end{align*}
			On   applying  the change of variable $\sin \tau \rightarrow \tau$, we arrive at
			\begin{align} \label{ int  }
				\frac{c_2\mathfrak{a}^{9/2}w^{3/2}}{\mathfrak{b}^5}	\int_{b_1/2}^{2b_2} \frac{1}{ \tau^5\sqrt{1-\tau^2} } U_{u,t}\left(\left(\frac{4 \pi \mathfrak{a}w^{1/3}}{3\mathfrak{b}  \tau}\right)^6\right)e\left( \frac{(k-1)\sin^{-1}\tau}{2\pi}+\frac{16\pi^2 \mathfrak{a}^3w}{27\mathfrak{b}^2\tau^2}\right)	 \mathrm{d}\tau.
			\end{align}
			Next we apply the  second derivative bound to the above integral. 
			Here the phase function is given by
			$$f(\tau)= \frac{(k-1)\sin^{-1}\tau}{2\pi}+\frac{16\pi^2 \mathfrak{a}^3w}{27\mathfrak{b}^2\tau^2}.$$
			On computing the first and the  second order derivatives, we see that 
			\begin{align}\label{derivative of f}
				f^{\prime}(\tau)&=\frac{(k-1)}{2\pi\sqrt{1-\tau^2}}-\frac{32\pi^2 \mathfrak{a}^3w}{27\mathfrak{b}^2\tau^3},  \\
				f^{\prime \prime}(\tau)&=\frac{(k-1)\tau}{2\pi{(1-\tau^2)^{3/2}}}+\frac{32\pi^2 \mathfrak{a}^3w}{9\mathfrak{b}^2\tau^4} \gg \frac{ \mathfrak{a}^3}{\mathfrak{b}^2\tau^4} \gg \frac{\mathfrak{b}^2}{\mathfrak{a}}. \notag
			\end{align}
			Thus on applying Lemma \ref{derivative bound} to \eqref{ int }, it  is bounded  above by  
			\begin{align*}
				\frac{   \text{Var}\, g+\max |g| }{\min \sqrt{f^{\prime \prime}(\tau)}} \ll    \frac{k^\epsilon  \mathfrak{a}^{9/2} }{\mathfrak{b}^{5} (\mathfrak{a}/\mathfrak{b})^5  \sqrt{\mathfrak{b}^2/\mathfrak{a}}}=\frac{k^{\epsilon}}{\mathfrak{b}},
			\end{align*}  
			where  $  \text{Var}\, g$ denotes the total variation of the weight function 
			$$g(\tau)= \frac{c_2\mathfrak{a}^{9/2}w^{3/2} U_{u,t}\left(\left({4 \pi \mathfrak{a}w^{1/3}}/{3\mathfrak{b}  \tau}\right)^6\right)}{\mathfrak{b}^5\tau^5\sqrt{1-\tau^2}} .$$ Hence, $	\mathrm{I}_{\pm}^{(1)} \ll {k^\epsilon}/{\mathfrak{b}}.$ On analyzing other $\mathrm{I}_{\pm}^{(i)}$'s in a similar fashion, we get $$	\mathrm{I}_{\pm} =	\mathrm{I}_{\pm}(m,N_0w,q) \ll {k^\epsilon}/{\mathfrak{b}}.$$ 
			
			Now we proceed to prove \eqref{y-integral expansion}.  We will give details for $\mathrm{I}_{\pm}^{(1)}$ only, as the analysis for other $\mathrm{I}_{\pm}^{(i)}$ is similar.  
			Let $q \sim C \gg k^{1+\epsilon}$. 	Note that this condition  assures that
			$\mathfrak{b} \asymp k$, as, by  \eqref{N0},  we have 
			\begin{align} \label{b=k}
				{k^{-\epsilon}(k-1)^2C^2}/{N}\ll M_1 \ll  k^{\epsilon}\max\left({(k-1)^2C^2}/{N},\,  T\right) \ll k^\epsilon (k-1)^2C^2/N,
			\end{align}
			and hence \begin{align}  \label{a<b}
				\mathfrak{a}=\frac{3(N{\tilde{N}})^{1/3}}{qr^{1/3}} \ll \frac{(N{N_0})^{1/3}}{qr^{1/3}}\ll  (kT)^{1/2}=k^{1-\eta/2}<k \asymp \mathfrak{b},
			\end{align} as $T=k^{1-\eta}<k.$ We  now   apply the  stationary phase analysis to  \eqref{ int }.  The stationary point of the phase function $f(\tau)$ occurs at $\tau_0$, where  $\tau_0$ satisfies
			$$\frac{(k-1)}{2\pi\sqrt{1-\tau_0^2}}=\frac{32\pi^2 \mathfrak{a}^3w}{27\mathfrak{b}^2\tau_0^3} \iff 
			\frac{\tau_0^3}{\sqrt{1-\tau_0^2}}=\left(\frac{4\pi}{3}\right)^3\frac{\mathfrak{a}^3w}{\mathfrak{b}^2(k-1)}.$$  Simplyfying it further, we see that $\tau_0$ satisfies  $$ \tau^6-\mathfrak{c}^2(1-\tau^2)=0,$$ where $\mathfrak{c}=\mathfrak{c}(w):=\left(\frac{4\pi}{3}\right)^3\frac{\mathfrak{a}^3w}{\mathfrak{b}^2(k-1)}.$ Upon  letting $\tau^2=\tau_1$, the above equation reduces to the cubic polynomial equation  $ \tau_1^3-\mathfrak{c}^2(1-\tau_1)=0,$
			which can be  solved using   Cardano's method. In fact, 	as  the discriminant of the cubic is negative, it has only one real root which can be found as follows: Let $\theta_1+\theta_2$ be the real  root. Upon substituting it into the cubic, we get 
			$$\theta_1^3+\theta_2^3+(3\theta_1\theta_2+\mathfrak{c}^2)(\theta_1+\theta_2)-\mathfrak{c}^2=0,$$
			which leads to the following system of equations:
			$$3\theta_1\theta_2+\mathfrak{c}^2=0, \ \ \ \theta_1^3+\theta_2^3-\mathfrak{c}^2=0.$$
	Now using the formula 
			$$(\theta_1^3-\theta_2^3)^2=(\theta_1^3+\theta_2^3)^2-4\theta_1^3\theta_2^3,$$
			we see that   the   real root  $\theta_1+\theta_2$ is given by  
			$$\sqrt[3]{\frac{\mathfrak{c}^2}{2}+\sqrt{\frac{\mathfrak{c}^4}{4}+\frac{\mathfrak{c}^6}{27}}}+\sqrt[3]{\frac{\mathfrak{c}^2}{2}-\sqrt{\frac{\mathfrak{c}^4}{4}+\frac{\mathfrak{c}^6}{27}}}.$$
		Hence we get 
			\begin{align}\label{binomial expansion}
				\tau_0=\tau_0(w)&=\left(\sqrt[3]{\frac{\mathfrak{c}^2}{2}+\sqrt{\frac{\mathfrak{c}^4}{4}+\frac{\mathfrak{c}^6}{27}}}+\sqrt[3]{\frac{\mathfrak{c}^2}{2}-\sqrt{\frac{\mathfrak{c}^4}{4}+\frac{\mathfrak{c}^6}{27}}}\right)^{1/2} \\
				&=\sqrt[6]{\frac{\mathfrak{c}^2}{2}+\sqrt{\frac{\mathfrak{c}^4}{4}+\frac{\mathfrak{c}^6}{27}}}\left(1-\frac{3}{\mathfrak{c}^2}\left(\sqrt{\frac{\mathfrak{c}^4}{4}+\frac{\mathfrak{c}^6}{27}}-\frac{\mathfrak{c}^2}{2}\right)^{2/3} \right)^{1/2}.
			\end{align}
			Now expanding the above expression using the binomial theorem, we see that 
			\begin{align}\label{stat point}
				\tau_0=\tau_0(w)=c_1\mathfrak{h}(w)+c_3(\mathfrak{h}(w))^3+c_3(\mathfrak{h}(w))^5...+c_{2n-1}(\mathfrak{h}(w))^{2n-1}+...,
			\end{align} 
			where $c_i $'s, $i=1, \,3, \, 5, \, \cdots $, are   some non-zero explicit absolute  constants and $$\mathfrak{h}(w)=\frac{\mathfrak{a}w^{1/3}}{\mathfrak{b}^{2/3}(k-1)^{1/3}}.$$ Note that the above series  in \eqref{stat point} is convergent  and each binomial expansion in \eqref{binomial expansion} is justified as $\mathfrak{c} \ll \mathfrak{a}^3 /(\mathfrak{b}^2(k-1)) \ll  k^{-3\eta/2}$.  Next we analyse the higher order derivatives of the phase function $f(\tau)$. On using \eqref{derivative of f} and  computing other higher order derivatives of $f(\tau)$, we get
			\begin{align*}
				&	f^{\prime \prime}(\tau) \asymp \mathfrak{b}^2/\mathfrak{a}= \mathfrak{a}(\mathfrak{a}/\mathfrak{b})^{-2}, \ \  f^{\prime}(\tau) \ll \mathfrak{a}(\mathfrak{a}/\mathfrak{b})^{-1},   \\
				&f^{(j)}(\tau)=\frac{(k-1)}{2\pi}\frac{\mathrm{d}^{j-2}}{\mathrm{d}\tau^{j-2}}\frac{\tau}{(1-\tau^2)^{3/2}}+\frac{32\pi^2 \mathfrak{a}^3w}{9\mathfrak{b}^2}\frac{\mathrm{d}^{j-2}(\tau^{-4})}{\mathrm{d}\tau^{j-2}}  \ll \mathfrak{a}(\mathfrak{a}/\mathfrak{b})^{-j},   \ j=3,\,4,...,
			\end{align*}
			where we used the fact $\mathfrak{a} \ll \mathfrak{b} \asymp k$ and 
			$$\frac{\mathrm{d}^{j-2}}{\mathrm{d}\tau^{j-2}}\frac{\tau}{(1-\tau^2)^{3/2}} \ll_j 1.$$ 
			On computing   derivatives of the weight function  $$g(\tau)= \frac{c_2\mathfrak{a}^{9/2}w^{3/2} U_{u,t}\left(\left({4 \pi \mathfrak{a}w^{1/3}}/{3\mathfrak{b}  \tau}\right)^6\right)}{\mathfrak{b}^5\tau^5\sqrt{1-\tau^2}},$$
			since $\tau \asymp \mathfrak{a}/\mathfrak{b}$, we see that 
			$$g^{(i)}(\tau) \ll{\mathfrak{a}^{-1/2}} \left({\mathfrak{a}}/{\mathfrak{b}}\right)^{-i}, \ \ \ i=0,1,2,...$$ 
			Thus, on applying Lemma \ref{stationaryphase} with $X=\mathfrak{a}^{-1/2}$, $Q=U=\mathfrak{a}/\mathfrak{b}$ and $Y=\mathfrak{a}$ to the $\tau$-integral in \eqref{ int }, we get \eqref{y-integral expansion}.

			\paragraph{Case 3}  $k^{-\epsilon} \mathfrak{b} \leq \mathfrak{a} \leq k^{\epsilon} \mathfrak{b}$. 
		In this case we can assume that $b_1/2<1,$  otherwise, we get back to the starting point of the discussion in Case 1. Consider  
			\begin{align}
				\mathrm{I}_{\pm}^{(1)}=\frac{1}{2\pi}\int_{0}^{\pi/2}e ^{i(k-1)\tau} \int_{1/2}^{5/2}U_{u,t}(y)e\left( \pm \mathfrak{a}w^{1/3}y^{1/3}-{\mathfrak{b}\sqrt{y}\sin \tau}/{2\pi}\right)\mathrm{d}y\, \mathrm{d}\tau.
			\end{align}
			We split the $\tau$-integral as follows:
			$$\int_{0}^{\pi/2} ... \, \mathrm{d}\tau=\int_{0}^{\sin^{-1}(b_1/2)} ... \, \mathrm{d}\tau+\int_{\sin^{-1}(b_1/2)}^{\pi/2} ... \, \mathrm{d}\tau.$$
			The first integral on the right side is negligibly small due to  absence of the stationary point.  Consider the second integral which is given by
			\begin{align} 
				\int_{\sin^{-1}(b_1/2)}^{\pi/2}e ^{i(k-1)\tau} 	\int_{\sqrt[3]{1/2}}^{\sqrt[3]{5/2}}3y^2U_{u,t}(y^3)e\left(  \mathfrak{a}w^{1/3}y-{\mathfrak{b}y^{3/2}\sin \tau}/{2\pi}\right)\textrm{d}y\, \mathrm{d}\tau.
			\end{align}
			On analyzing the $y$-integral like Case 2, we get the lemma.
			
		\end{proof} 
		\begin{proof}[{ \bf Proof of  Proposition \ref{bound for integral}}]
			Recall from \eqref{y integral without u} that 
			$$ \mathrm{I}_{\pm}(m,\tilde{N}w,q)=\int_{1/2}^{5/2}\, U_{u,t}(y)\, 
			e\left(\pm \mathfrak{a}w^{1/3}y^{1/3}\right)J_{k-1}\left( \mathfrak{b} \sqrt{y}\right)\, \mathrm{d}y.$$
			Note that $$ \frac{\partial^j}{\partial w^j}\mathrm{I}_{\pm}(m,\tilde{N}w,q) \ll\mathfrak{a}^j,  \ \ \ j\geq 0.$$
			Similarly it follows that 
			$$ \frac{\partial^j}{\partial w^j}\mathrm{I}_{\pm}(m^\prime,\tilde{N}w,q^\prime) \ll\mathfrak{a}^{\prime j},\ \ \ \ j\geq0.$$
			Hence, on applying integration by parts $j$-times to the $w$-integral in \eqref{J}, we see that 
			$$	\mathfrak{J}_{\pm} \ll (k^\epsilon +\mathfrak{a}+\mathfrak{a}^\prime)^j \left(\frac{q_2q_2^\prime q_1rn_1}{n_2\tilde{N}}\right)^j \ll \left( \frac{(N\tilde{N})^{1/3}}{Cr^{1/3}}\right)^j\left(\frac{C^2rn_1}{q_1n_2\tilde{N}}\right)^j=\left(\frac{N^{1/3} Cr^{2/3}n_1}{q_1n_2\tilde{N}^{2/3}}\right)^j. $$
			Thus $	\mathfrak{J}_{\pm}$ is negligibly small if 
			$$\frac{N^{1/3} Cr^{2/3}n_1}{q_1n_2\tilde{N}^{2/3}} \ll \frac{1}{k^\epsilon} \iff  	n_2 \gg k^{\epsilon}\frac{CN^{1/3}r^{2/3}n_1}{q_1\tilde{N}^{2/3}}.$$ 
			Next we prove  $	\mathfrak{J}_{\pm} \ll {k^\epsilon C^2}/{(M_1N)}$. 
			
			\paragraph{Case 1}$\mathfrak{a} \not\asymp \mathfrak{b} $, i.e.,  $\mathfrak{a}^\prime \asymp \mathfrak{a}  \ll k^{-\epsilon} \mathfrak{b} \asymp k^{-\epsilon} \mathfrak{b}^\prime$ or $\mathfrak{a}^\prime \asymp \mathfrak{a} \gg k^{\epsilon} \mathfrak{b} \asymp k^\epsilon \mathfrak{b}^\prime$.
			\noindent
			In the case when $ \mathfrak{a} \gg k^{\epsilon} \mathfrak{b}$,  on  applying  Lemma \ref{y-integral bound}  to $\mathrm{I}_{\pm}(m,\tilde{N}w,q)$, we see that $\mathfrak{J}_{\pm}$ is negligibly small.   
			In the other case, i.e., $\mathfrak{a}^\prime \asymp \mathfrak{a}  \ll k^{-\epsilon} \mathfrak{b} \asymp k^{-\epsilon} \mathfrak{b}^\prime$, on applying Lemma \ref{y-integral bound}  to \eqref{J}, we get 
			\begin{align}
				\mathfrak{J}_{\pm}\ll  \int_{\mathbb{R}}\phi(w)\, |\mathrm{I}_{\pm}(m,\tilde{N}w,q)|\, |\overline{\mathrm{I}_{\pm}(m^\prime,\tilde{N}w,q^\prime)}|\, \mathrm{d}w \ll \frac{k^\epsilon}{\mathfrak{b}\mathfrak{b}^\prime} \ll \frac{k^\epsilon C^2}{M_1N}.
			\end{align}
			\paragraph{Case 2} $\mathfrak{a} \asymp \mathfrak{b} $, i.e.,  $k^{-\epsilon} \mathfrak{b} \leq \mathfrak{a} \leq k^{\epsilon} \mathfrak{b}$. 
				On applying the last part of Lemma \ref{y-integral bound}  to \eqref{J}, we see that 
			\begin{align}
				\mathfrak{J}_{\pm}   \ll	\frac{(\mathfrak{a}\mathfrak{a}^\prime)^{9/2}}{(\mathfrak{b}\mathfrak{b}^\prime)^5} &\int_{b_1/2}^{1}	\int_{b_1^\prime/2}^{1} \frac{1}{ \tau^5\sqrt{1-\tau^2} } \frac{1}{ \tau^{\prime 5}\sqrt{1-\tau^{\prime 2}} } \\
				&\times \Big| \int_{2/3}^3g_3(\tau,\tau^\prime,w)e\left(wf_3(\tau,\tau^\prime)	\right)\,  \mathrm{d}w\Big|\, \mathrm{d}\tau\,\mathrm{d}\tau^\prime, \notag
			\end{align}
			where $$f_3(\tau,\tau^\prime)= \frac{16\pi^2 \mathfrak{a}^3}{27\mathfrak{b}^2\tau^2}-\frac{16\pi^2 \mathfrak{a}^{\prime 3}}{27\mathfrak{b}^{\prime 2}\tau^{\prime 2}}-\frac{ n_2 \tilde{N}}{q_2q_2^\prime q_1r n_1}$$ 
			and $$g_3(\tau,\tau^\prime,w)= \phi(w)w^3 U_{u,t}\left(\left(\frac{4 \pi \mathfrak{a}w^{1/3}}{3\mathfrak{b}  \tau}\right)^6\right)\bar{U}_{u^\prime,t^\prime}\left(\left(\frac{4 \pi \mathfrak{a}^\prime w^{1/3}}{3\mathfrak{b}^\prime  \tau^\prime}\right)^6\right).$$
			On applying the change of variable  $\tau \rightarrow \ 1/\sqrt{\tau}$, $\tau^\prime \rightarrow \ 1/\sqrt{\tau^\prime}$, we arrive at 
			\begin{align}\label{a b same }
				\mathfrak{J}_{\pm}   \ll	\frac{(\mathfrak{a}\mathfrak{a}^\prime)^{9/2}}{(\mathfrak{b}\mathfrak{b}^\prime)^5} &\int_1^{4/b_1^2}	\int_1^{4/b_1^{\prime 2}} \frac{\tau^{3/2}}{2 \sqrt{\tau-1} } \frac{\tau^{\prime 3/2}}{ 2\sqrt{\tau^\prime-1} } \\
				&\times \Big| \int_{2/3}^3g_3(1/\sqrt{\tau},1/\sqrt{\tau^\prime},w)e\left( \frac{16\pi^2 \mathfrak{a}^3w}{27\mathfrak{b}^2} f_4({\tau},{\tau^\prime})	\right)\,  \mathrm{d}w\Big|\, \mathrm{d}\tau\,\mathrm{d}\tau^\prime, \notag
			\end{align}
			where $$f_4(\tau,\tau^\prime)=\tau-\frac{\mathfrak{a}^{\prime 3}\mathfrak{b}^2}{\mathfrak{a}^3\mathfrak{b}^{\prime 2}}\tau^\prime-\frac{ 27n_2 \tilde{N} \mathfrak{b}^2}{16 \pi^2q_2q_2^\prime q_1r n_1\mathfrak{a}^3}.$$
			Now using the change of variable 
			$$ \frac{\mathfrak{a}^{\prime 3}\mathfrak{b}^2}{\mathfrak{a}^3\mathfrak{b}^{\prime 2}}\tau^\prime+\frac{ 27n_2 \tilde{N} \mathfrak{b}^2}{16 \pi^2q_2q_2^\prime q_1r n_1\mathfrak{a}^3} \rightarrow \tau^\prime,$$
			we arrive at the following    $w$-integral 
			$$ \int_{2/3}^3 g_3(...,w)e\left(w \frac{16\pi^2 \mathfrak{a}^3}{27\mathfrak{b}^2}(\tau-\tau^\prime)	\right)\,  \mathrm{d}w,$$
			where  $g_3(...,w)$ is given by 
			$$ \phi(w)w^3 {U}_{u,t}\left(  \frac{(4 \pi \mathfrak{a})^6 \tau^3 w^2}{(3\mathfrak{b})^6  } \right)\bar{U}_{u^\prime,t^\prime}\left(  \frac{(4 \pi \mathfrak{a}^\prime)^6  w^2}{(3\mathfrak{b}^\prime)^6  } \left( \frac{\mathfrak{a}^{ 3}\mathfrak{b}^{\prime 2}}{\mathfrak{a}^{\prime 3}\mathfrak{b}^{ 2}}\tau^\prime-\frac{ 27n_2 \tilde{N}  \mathfrak{b}^{\prime 2}}{16 \pi^2q_2q_2^\prime q_1r n_1\mathfrak{a}^{\prime 3}} \right)^3\right). $$
			Note that 
			\begin{align}\label{derivative of g3}
				 \frac{\partial^j}{\partial w^j} g_3(...,w)  \ll_j k^{\epsilon j},\ \ \ \ j\geq0,
			\end{align}
		as $\mathfrak{a} \asymp \mathfrak{b} $  and 
		$$\frac{\mathfrak{a}^{ 3}\mathfrak{b}^{\prime 2}}{\mathfrak{a}^{\prime 3}\mathfrak{b}^{ 2}}\tau^\prime- \frac{ 27n_2 \tilde{N}  \mathfrak{b}^{\prime 2}}{16 \pi^2q_2q_2^\prime q_1r n_1\mathfrak{a}^{\prime 3}} \ll k^\epsilon +\frac{(\mathfrak{a}+\mathfrak{a}^\prime) \mathfrak{b}^{\prime 2}}{\mathfrak{a}^{\prime 3}} \ll k^\epsilon,$$
	where, in the first inequality, we used $ \frac{ n_2 \tilde{N}  }{q_2q_2^\prime q_1r n_1} \ll \mathfrak{a}+\mathfrak{a}^\prime$, which follows by applying  integration by parts to the $w$-integral in \eqref{J}. 
			On applying integration by parts repeatedly, we see that the above integral is negligibly small unless 
			$$|\tau-\tau^\prime| \ll k^\epsilon\mathfrak{b}^2/\mathfrak{a}^3.$$
			Now writing  $\tau-\tau^\prime =\tau_2$, with $\tau_2 \ll  k^\epsilon\mathfrak{b}^2/\mathfrak{a}^3$, and estimating all the integrals  in \eqref{a b same }  trivially, we get 
			$$	\mathfrak{J}_{\pm}\ll \frac{(\mathfrak{a}\mathfrak{a}^\prime)^{9/2}}{(\mathfrak{b}\mathfrak{b}^\prime)^5} \frac{k^\epsilon \mathfrak{b}^2 }{\mathfrak{a}^3} \ll \frac{1}{(\mathfrak{b}\mathfrak{b}^\prime)^{1/2}}\frac{k^\epsilon}{\mathfrak{b}}\ll \frac{k^\epsilon C^2}{M_1N},$$
			where we used the fact $\mathfrak{a}^\prime \asymp \mathfrak{a} \asymp \mathfrak{b} \asymp \mathfrak{b}^\prime
			$. Hence we get \eqref{integral bound for all case}. 
			
			Now   we proceed to prove the last part. Let $q\sim C \gg k^{1+\epsilon}$. We also have $q^\prime  \sim C \gg k^{1+\epsilon}$.  Note that  in this situation we have $\mathfrak{a}\ll k^{-\epsilon} \mathfrak{b}$,  $\mathfrak{a}^\prime\ll k^{-\epsilon} \mathfrak{b}^\prime$ and $\mathfrak{b} \asymp \mathfrak{b}^\prime \asymp k$ (see \eqref{b=k} and \eqref{a<b}).  On substituting the main term of  $\mathrm{I}_{\pm}(m,\tilde{N}w,q)$ from  \eqref{y-integral expansion} and a similar expression for  $\mathrm{I}_{\pm}(m^{\prime},\tilde{N}w,q^{\prime})$    into  \eqref{J}, we  arrive at the following expression:
			\begin{align}\label{w integral sorted}
				&\frac{c_3^2 (\mathfrak{a}\mathfrak{a}^\prime)^{9/2}}{(\mathfrak{b}\mathfrak{b}^\prime)^5}\int_{\mathbb{R}} \phi_1(w) e\left(f_5(w)\right)\mathrm{d}w,
			\end{align}
			where \begin{align}\label{phi 1}
				\phi_1(w)=& \frac{ 1 }{\sqrt{f^{\prime \prime}(\tau_0)}}\frac{1}{\tau_0^5 \sqrt{1-\tau_0^2}}\frac{1 }{\sqrt{f_2^{\prime \prime}(\tau_0^\prime)}}\frac{1}{\tau_0^{\prime 5} \sqrt{1-\tau_0^{\prime 2}}}\, \\&\times U_{u,t}\left(\left(\frac{4\pi \mathfrak{a}w^{1/3}}{3\mathfrak{b}  \tau_0}\right)^6\right)\bar{U}_{u^\prime,t^\prime}\left(\left(\frac{4\pi \mathfrak{a}^\prime w^{1/3}}{3\mathfrak{b}^\prime  \tau_0^\prime}\right)^6\right),\notag
			\end{align}
			and $$f_5(w)= \frac{(k-1)(\sin^{-1}\tau_0-\sin^{-1}\tau_0^{\prime})}{2 \pi}+\frac{16\pi^2}{27}\left(\frac{\mathfrak{a}^3w}{\mathfrak{b}^2\tau_0^2}-\frac{\mathfrak{a}^{\prime 3}w}{\mathfrak{b}^{\prime 2}\tau_0^{\prime 2}}\right)-\frac{\tilde{N} n_2w}{q_2q_2^\prime q_1r n_1},$$
			to which we apply the third derivative bound. Recall from \eqref{stat point} that 
			\begin{align}\label{stat point 2}
				\tau_0=\tau_0(w)=c_1\mathfrak{h}(w)+c_3(\mathfrak{h}(w))^3+c_3(\mathfrak{h}(w))^5...+c_{2n-1}(\mathfrak{h}(w))^{2n-1}+...,
			\end{align} 
			with  $$\mathfrak{h}(w)=\frac{\mathfrak{a}w^{1/3}}{\mathfrak{b}^{2/3}(k-1)^{1/3}}, \ \  \mathfrak{b}=\frac{4\pi \sqrt{mN}}{q} \ \  \mathrm{and } \ \  \mathfrak{a}=\frac{3(N\tilde{N})^{1/3}}{qr^{1/3}}. $$ 
			and $\tau_0^\prime$ is similarly defined. On applying  the change of variable  $w \rightarrow w^3$ in \eqref{w integral sorted}, we see that the  phase function  is given by 
			\begin{align*}
				\frac{(k-1)(\sin^{-1}\tau_0(w^3)-\sin^{-1}\tau_0^{\prime}(w^3))}{2 \pi} 
				+\frac{16\pi^2}{27}\left( \frac{\mathfrak{a}^3w^3}{\mathfrak{b}^2\tau_0^2(w^3)}-\frac{ \mathfrak{a}^{\prime 3}w^3}{\mathfrak{b}^{\prime 2}\tau_0^{\prime 2}(w^3)}\right)-\frac{\tilde{N} n_2w^3}{q_2q_2^\prime q_1r n_1}.
			\end{align*}
			On applying  the Taylor series expansion of $	\sin^{-1}\tau_0(w^3)$,  we see that 
			\begin{align*}
				\sin^{-1}\tau_0(w^3)&=\tau_0(w^3)+{(\tau_0(w^3))^3}/{6}+\cdots \\
				&=d_1\mathfrak{h}(w^3)+d_3(\mathfrak{h}(w^3))^3+\cdots \\
				&=d_1\frac{\mathfrak{a}w}{\mathfrak{b}^{2/3}(k-1)^{1/3}}+d_3\frac{\mathfrak{a}^3w^3}{\mathfrak{b}^{2}(k-1)}+\cdots,
			\end{align*}
			where $d_1,\, d_3\, \cdots$ are some absolute constants. Thus 
			$$ \frac{\partial^3}{\partial w^3}	\sin^{-1}\tau_0(w^3) \ll \frac{\mathfrak{a}^3}{\mathfrak{b}^{2}(k-1)}. 
			$$
			Similarly, 
			$$\frac{\partial^3}{\partial w^3}	\sin^{-1}\tau_0^\prime(w^3) \ll \frac{\mathfrak{a}^{\prime 3}}{\mathfrak{b}^{ \prime 2}(k-1)}. 
			$$
			Next we consider $ {\mathfrak{a}^3w^3}/({\mathfrak{b}^2\tau_0^2(w^3)})$. On applying the Taylor series expansion, we get
			\begin{align*}
				\frac{\mathfrak{a}^3w^3}{\mathfrak{b}^2\tau_0^2(w^3)}=\frac{(k-1)(\mathfrak{h}(w^3))^3}{\tau_0^2(w^3)}	&=\frac{(k-1)\mathfrak{h}(w^3)}{c_1^2}\left(1+\frac{c_3(\mathfrak{h}(w^3))^3}{c_1\mathfrak{h}(w^3)}+\cdots\right)^{-2} \\
				&=\frac{(k-1)\mathfrak{h}(w^3)}{c_1^2}\left(1-\frac{2c_3(\mathfrak{h}(w^3))^3}{c_1\mathfrak{h}(w^3)}-\cdots\right) \\
				&=\frac{(k-1) }{c_1^2}\left(\mathfrak{h}(w^3)-\frac{2c_3(\mathfrak{h}(w^3))^3}{c_1}-\cdots\right) 
			\end{align*}
			Thus
			$$ \frac{\partial^3}{\partial w^3}	 \frac{\mathfrak{a}^3w^3}{\mathfrak{b}^2\tau_0^2(w^3)} \ll \frac{\mathfrak{a}^3}{\mathfrak{b}^2}.$$
			A similar analysis  also gives us 
			\begin{align}
			\frac{\partial^3}{\partial w^3} 	 \frac{\mathfrak{a}^{\prime 3}w^3}{\mathfrak{b}^{\prime 2}\tau_0^{\prime 2}(w^3)} \ll \frac{\mathfrak{a}^{\prime 3}}{\mathfrak{b}^{\prime 2}}.
			\end{align}
			Hence, upon combining the above estimates, we conclude that 
			$$ \frac{\mathrm{\partial^3}f_5(w^3)}{\partial w^3}= O\left( \frac{\mathfrak{a}^3}{\mathfrak{b}^2}+\frac{\mathfrak{a}^{\prime 3}}{\mathfrak{b}^{\prime 2}}\right)	-\frac{6\tilde{N}n_2}{q_2q_2^\prime q_1r n_1}.$$
			Since $n_2 \neq 0$, we note that 
			\begin{align*}
				\frac{\mathfrak{a}^3}{\mathfrak{b}^2}+\frac{\mathfrak{a}^{\prime 3}}{\mathfrak{b}^{\prime 2}} \ll \frac{N\tilde{N}}{C^3rk^2} \ll  \frac{(k^3/r^2)\tilde{N}}{C^2rk^{3+\epsilon}} \ll \frac{\tilde{N} }{k^\epsilon C^2r (n_1,r)} \ll \frac{\tilde{N} }{k^{\epsilon}(C^2/q_1)r n_1} \ll 	\frac{k^{-\epsilon}6\tilde{N} |n_2|}{q_2q_2^\prime q_1r n_1}.
			\end{align*}
			In the first inequality,  we  used the fact  $\mathfrak{a} \asymp \mathfrak{a}^\prime$, $\mathfrak{b} \asymp \mathfrak{b}^\prime \asymp k$. For the second inequality, we used  $Nr^2 \ll k^{3+\epsilon}$  and $C \gg k^{1+\epsilon} $, while for the second  last inequality, $(n_1,r) \geq  n_1/ q_1$, is being used.
			Hence we see that  
			$$ \Big|\frac{\partial^3  f_5(w^3)}{\partial w^3}\Big|= \Big|O\left( \frac{\mathfrak{a}^3}{\mathfrak{b}^2}+\frac{\mathfrak{a}^{\prime 3}}{\mathfrak{b}^{\prime 2}}\right)	-\frac{6\tilde{N}n_2}{q_2q_2^\prime q_1r n_1}\Big| \gg 	\frac{\mathfrak{a}^3}{\mathfrak{b}^2}+\frac{\mathfrak{a}^{\prime 3}}{\mathfrak{b}^{\prime 2}} \asymp \frac{N\tilde{N}}{C^3rk^2}.$$
			On computing the variation of $\phi_1(w)$ (see \eqref{phi 1}), we note  that 
			\begin{align}\label{bound for phi}
				\mathrm{Var}\,\phi_1\ll \frac{1}{\sqrt{ \mathfrak{b}^2/\mathfrak{a}}}\frac{1}{(\mathfrak{a}/k)^5}\frac{1}{\sqrt{ \mathfrak{b}^{\prime 2}/\mathfrak{a}^{\prime}}}\frac{1}{(\mathfrak{a}^\prime/k)^5} \ll \frac{1}{\mathfrak{b}^2/\mathfrak{a}}\frac{1}{(\mathfrak{a}/k)^{10}},
			\end{align}
			where we used  $ f^{\prime \prime }(\tau_0)\asymp \mathfrak{b}^2/\mathfrak{a}$,  $f_2^{ \prime \prime }(\tau_0^{\prime}) \asymp \mathfrak{b}^{\prime 2}/\mathfrak{a}^{\prime}$, $\tau_0 \asymp  \mathfrak{a}/(\mathfrak{b}^{2/3}(k-1)^{1/3}) \asymp \mathfrak{a}/k$ and $\tau_0^\prime \asymp  \mathfrak{a}^\prime/k$. 
			Hence, on applying the third derivative bound (see  Lemma \ref{derivative bound}) to \eqref{w integral sorted}, we see that \eqref{w integral sorted} is bounded by 
			$$ \frac{c_3^2 (\mathfrak{a}\mathfrak{a}^\prime)^{9/2}}{(\mathfrak{b}\mathfrak{b}^\prime)^5} \frac{\mathrm{Var}\, \phi_1 + +\max |\phi_1|}{\min{|f_5(w^3)|^{1/3}}} \ll \frac{ \mathfrak{a}^{9}}{\mathfrak{b}^{10}}\frac{1}{ \mathfrak{b}^2/\mathfrak{a}}\frac{1}{(\mathfrak{a}/k)^{10}} \frac{(C^3rk^2)^{1/3}}{(N\tilde{N})^{1/3}} \asymp \frac{Cr^{1/3}k^{2/3}}{k^2(N\tilde{N})^{1/3}}.$$
			
			Hence we get Proposition \ref{bound for integral}. 
		\end{proof}
		We conclude this section by giving the final estimation of  the main integral $\mathcal{J}_{\pm}$  defined in  \eqref{ integral} in the following corollary:
		\begin{corollary} \label{final int bound}
			Let $\mathcal{J}_{\pm}$ be the integral transform as   defined  in \eqref{ integral}. Then  we have 
			\begin{align}\label{final J bound for all case }
				\mathcal{J}_{\pm} \ll \frac{k^\epsilon C^4}{Q^2M_1N}. 
			\end{align}
			Furthermore, if $C \gg k^{1+\epsilon}$ and  $n_2\neq 0$,   we have	
			\begin{align}\label{final integral bound for large c}
				\mathcal{J}_{\pm} \ll \frac{k^\epsilon C^2}{Q^2}\frac{Cr^{1/3}k^{2/3}}{k^2(N\tilde{N})^{1/3}}.
			\end{align}
		\end{corollary}
		\begin{proof}
			Let's recall from \eqref{main integral after u} that 
			\begin{align*}
				\mathcal{J}_{\pm}=	\int_{\mathbb{R}}  \int_{\mathbb{R}}  V\left(\frac{t}{T}\right)  V\left(\frac{t^\prime}{T}\right) \int_{u \ll \frac{k^\epsilon C}{QT}} \,   \int_{u^\prime \ll \frac{k^\epsilon C}{QT}} \,\mathrm{I}_u\,  \overline{ \mathrm{I}}_{u^\prime}  \,	\mathfrak{J}_{\pm} \,  \mathrm{d}u^\prime  \,  \mathrm{d}u \,  \mathrm{d} t^\prime   \,  \mathrm{d} t+O(k^{-2020}),
			\end{align*}
			where \begin{align*}
				\mathrm{I}_{u}=\int_{\mathbb{R}}W(x/Q^{\epsilon})g(q,x)e\left(\frac{Nxu}{qQ}\right)\, \mathrm{d}x,
			\end{align*}
			and $	\mathrm{I}_{u^\prime}$ is similarly defined. On applying  the bound $	\mathfrak{J}_{\pm} \ll {k^\epsilon C^2}/{(M_1N)}$ from Proposition \ref{bound for integral}, we see  that  
			\begin{align}\label{J 4}
				|\mathcal{J}_{\pm}| \ll \frac{k^\epsilon C^2}{M_1N}\int_{\mathbb{R}}  \int_{\mathbb{R}}  V\left(\frac{t}{T}\right)  V\left(\frac{t^\prime}{T}\right) \int_{u \ll \frac{k^\epsilon C}{QT}} \,   \int_{u^\prime \ll \frac{k^\epsilon C}{QT}} \,|\mathrm{I}_u|\,  |\overline{ \mathrm{I}}_{u^\prime}|  \,	 \mathrm{d}u^\prime  \,  \mathrm{d}u \,  \mathrm{d} t^\prime   \,  \mathrm{d} t.
			\end{align}
		Note that 
			\begin{align*}
				\int_{u \ll \frac{k^\epsilon C}{QT}}\,|\mathrm{I}_u|\,  \mathrm{d}u  \ll \int_{u \ll \frac{k^\epsilon C}{QT}}  \,\int_{\mathbb{R}}W(x/Q^{\epsilon})|g(q,x)|\mathrm{d}x\,   \mathrm{d}u \ll  \frac{k^\epsilon C}{QT}Q^\epsilon ,
			\end{align*}
			where we used Property 4 (see \eqref{g properties}) of $g(q,x)$. The same bound holds for the $u^\prime$-integral as well.  Thus, on plugging these bounds into \eqref{J 4} and estimating the $t$ and $t^\prime$-integral trivially, we get \eqref{final J bound for all case }.  On analysing the $u$, $u^\prime$, $t$ and $t^\prime$-integrals as above and applying the bound \eqref{y integral  bound for large c} from  Proposition \ref{bound for integral}, we get the second part of the corollary. 
		\end{proof}

		\section{\bf  Analysis of   the  zero  frequency: $n_2=0$}
		With all the ingredients in   hand, we  now give final  estimates for  $S_r(N)$,  given  in  \eqref{S(N)after cauchy},  in the present and coming sections. The zero frequency case, i.e.,  $n_2=0$, needs to be analysed differently.   Let $\Omega_{\pm}^{0}$  denote the contribution   of the zero frequency to  $\Omega_{\pm}$, given in \eqref{final omega},  and let $S_r^{0}(N)$ denote the contribution of  $\Omega_{\pm}^{0}$ to  $S_r(N)$.  We have  the following lemma:
		\begin{lemma} \label{SN bound for zero frq}
			Let $\Omega_{\pm}^{0}$ and  $S_r^{0}(N)$ be defined as above.  Then we have
			\begin{align*}
				\Omega_{\pm}^{0} \ll  \frac{k^\epsilon N_0C^6r}{q_1n_1^2Q^2N}  (C+M_1),
			\end{align*}
			and 
			\begin{align*}
				S_r^{0}(N)\ll k^\epsilon r^{1/2}N^{1/2}k^{3/2-\eta/2},
			\end{align*}
			where $T=k^{1-\eta}$.
		\end{lemma}
		\begin{proof}
			Let's recall  from \eqref{Omega} that 
			\begin{align}\label{zero frequency omega}
				\Omega_{\pm}^{0} \ll &k^{\epsilon} \sup_{\tilde{N} \ll N_0}\frac{\tilde{N}}{n_1^2}\mathop{\sum \sum}_{q_2,\, q_2^{\prime}\sim C/q_1}\mathop{\sum \sum}_{m,\,m^{\prime} \sim M_1}  |\lambda_f(m)||\lambda_{f}(m^\prime)|   |\mathfrak{C}_{\pm}|\,|\mathcal{J}_{\pm}|.
			\end{align}
			Consider the congruence condition
			$$\pm \bar{\alpha}q_2^\prime \mp\bar{\alpha}^\prime q_2\, \equiv\, n_2  \ {\rm mod} \ q_1q_2q_2^\prime r/n_1$$
			given in the expression  \eqref{character sum} of $\mathfrak{C}_{\pm}$.
			For $n_2=0$, it follows that    $q_2=q_2^{\prime}$ and $\alpha=\alpha^{\prime}$. Hence  we get 
			\begin{align*}
				\mathfrak{C}_{\pm}&=\mathop{\sum \sum}_{\substack{d,\, d^\prime |q }}dd^\prime\mu\left(\frac{q}{d}\right)\, \mu\left(\frac{q}{d^\prime}\right)  \mathop{\sideset{}{^\star} \sum_{\alpha \; {\rm mod} \; qr/n_1 } }_{\substack{n_1\alpha\equiv-m \, {\rm mod} \, d \\ n_1\alpha \equiv-m^\prime \, {\rm mod} \, d^\prime  }} 1\\
				&\ll  \mathop{\sum \sum}_{\substack{d, \, d^{\prime} \vert q \\ (d,d^{\prime})|(m-m^{\prime})}}dd^\prime\,  \frac{qr}{n_1[d/(n_1,d),\, d^\prime/(n_1,d^\prime)]} \ll \mathop{\sum \sum}_{\substack{d,\, d^{\prime} \vert q \\ (d,d^{\prime})|(m-m^{\prime})}} dd^\prime \frac{qr}{[d,d^{\prime}]}.
			\end{align*}
			On plugging  the above expression and the bound  $	\mathcal{J}_{\pm} \ll {k^\epsilon C^4}/{(Q^2M_1N)} $ from Corollary \ref{final int bound} into \eqref{zero frequency omega}, we get
			\begin{align*}
				\Omega_{\pm}^{0} & \ll \frac{k^\epsilon C^4}{Q^2M_1N} \sup_{\tilde{N} \ll N_0}\frac{\tilde{N}}{n_1^2}\, \mathop{\sum }_{\substack{q_2 \sim C/q_1 }} qr \mathop{\sum \sum}_{\substack{d, \, d^{\prime} \vert q}}(d,\, d^{\prime}) \mathop{\sum \sum}_{\substack{m,m^\prime  \sim M_{1} \\ (d,\,d^{\prime})|(m-m^{\prime}) }} |\lambda_f(m)||\lambda_{f}(m^\prime)|. 
			\end{align*}
		We use the inequality 
		\begin{align}\label{fou coef}
			 |\lambda_f(m)||\lambda_{f}(m^\prime)|  \leq \frac{1}{2} (|\lambda_{f}(m)|^2+ |\lambda_{f}(m^\prime)|^2)
		\end{align}
		 to count the number of $m$ and $m^\prime$ as follows: 
		\begin{align*}
			 \mathop{\sum \sum}_{\substack{m,m^\prime  \sim M_{1} \\ (d,\,d^{\prime})|(m-m^{\prime}) }} |\lambda_f(m)\lambda_{f}(m^\prime)| &\ll \sum_{m\sim M_1} |\lambda_{f}(m)|^2+  \mathop{\sum \sum}_{\substack{m,m^\prime  \sim M_{1} \\ (d,\,d^{\prime})|(m-m^{\prime}), \, m\neq m^\prime }} (|\lambda_f(m)|^2+|\lambda_{f}(m^\prime)|^2) \\
			 &\ll k^\epsilon M_1+ \sum _{m\sim M_1}  |\lambda_f(m)|^2 \mathop{ \sum}_{\substack{m^\prime  \sim M_{1} \\ (d,\,d^{\prime})|(m-m^{\prime}), \, m^\prime\neq m }} 1 \\
			 &\ll k^\epsilon M_1(1+M_1/(d,d^\prime)),
		\end{align*}
	where we used the Ramanujan bound on average (see \eqref{RS for holomorphic} and \eqref{RS bound}).  Thus we see that 
		\begin{align*}
		\Omega_{\pm}^{0} 
		& \ll \frac{k^\epsilon N_0C^4}{n_1^2Q^2M_1N}\mathop{\sum }_{\substack{q_2 \sim C/q_1 }} qr  \mathop{\sum \sum}_{\substack{d ,\, d^{\prime} \vert q}}(M_1(d,d^{\prime})+M_1^2) \\
		& \ll \frac{k^\epsilon N_0C^4}{n_1^2Q^2M_1N}\mathop{\sum }_{\substack{q_2 \sim C/q_1 }} qr  (M_1q+M_1^2)   \ll \frac{k^\epsilon N_0C^6r}{q_1n_1^2Q^2N}  (C+M_1).
	\end{align*}
Hence we have the  first part of the lemma. On substituting  the above bound in place of $\Omega_{\pm}$ in \eqref{S(N)after cauchy}, we get 
			\begin{align*}
				S_r^{0}(N)\ll \mathop{\sup_{ \substack{M_1\ll M_0 \\ C \ll Q}}}\frac{N^{5/3+\epsilon}}{QTr^{2/3}C^3} \sum_{\frac{n_1}{(n_1,r)}\ll C}n_1^{1/3}\Theta^{1/2}\sum_{\frac{n_1}{(n_1,r)}|q_1|(n_1r)^\infty}\frac{C^{3} {(N_0r)}^{1/2}}{n_1q_1^{1/2}Q\sqrt{N}}\left(\sqrt{M_1}+\sqrt{C}\right).
			\end{align*}
			Executing the $q_1$-sum trivially and replacing the range for $n_1$ by the longer range $n_1 \ll Cr$, we get 
			\begin{align*}
				S_r^0(N) \ll k^\epsilon \mathop{\sup_{ \substack{M_1\ll M_0 \\ C \ll Q}}}\frac{N^{2/3}(N_0r)^{1/2}}{r^{2/3}\sqrt{N}}\sum_{n_1\ll Cr}\frac{(n_1,r)^{1/2}}{n_1^{7/6}}\Theta^{1/2}\left(\sqrt{M_1}+\sqrt{C}\right).
			\end{align*}
			Next we evaluate the  $n_1$-sum, using  Cauchy's inequality  and the Ramanujan bound on average (see Lemma \ref{ramanubound}), as follows:
			\begin{align}\label{theta bound}
				\sum_{n_1\ll Cr}\frac{(n_1,r)^{1/2}}{n_1^{7/6}}\Theta^{1/2} \ll \left[\sum_{n_1 \ll Cr}\frac{(n_1,r)}{n_1}\right]^{1/2}\left[\mathop{\sum \sum}_{n_{1}^{2} n_{2} \leq N_0} \frac{\vert \lambda_{\pi}(n_{1},n_{2})\vert ^{2} }{(n_1^2n_2)^{2/3}}\right]^{1/2}\ll_{\pi, \epsilon} N_0^{1/6+\epsilon}.
			\end{align}
			Thus we arrive at 
			\begin{align}\label{zero s}
				S_r^0(N) \ll k^\epsilon\frac{N^{2/3}N_0^{2/3}}{r^{1/6}\sqrt{N}}\left(\sqrt{M_0}+\sqrt{Q}\right).
			\end{align}
			Note that
				 $$Q=k^\epsilon \sqrt{N/{T}} \ll k^{3/2+\epsilon}/\sqrt{{T}} \ll k^{2+\epsilon}/{T}\ll k^{2+\epsilon}Q^2/N.$$
			We also have $M_0= k^{\epsilon}\max\left({(k-1)^2C^2}/{N},\,  T\right)\ll k^{2+\epsilon} Q^2/N $ and  $$N_0=k^{\epsilon}\max \left\{{\left(CT\right)^{3}r}/{N},\,  T^{3/2} N^{1/2}r \right\} \ll k^\epsilon (QT)^3r/N\ll k^\epsilon T^{3/2}\sqrt{N}r.$$
			Finally, upon  using the above bounds in \eqref{zero s},   we get
			\begin{align*}
				S_r^0(N) \ll \frac{k^\epsilon r^{2/3} TN}{r^{1/6}\sqrt{N}}\frac{kQ}{\sqrt{N}} \ll k^\epsilon r^{1/2}N^{1/2}k^{3/2-\eta/2}.
			\end{align*}
			Hence the lemma follows.
		\end{proof}
		\section{\bf Analysis  of the  non-zero frequencies: $n_2 \neq 0$}
		It now remains to estimate  $S_r(N)$  corresponding to the non-zero frequencies, i.e., $n_2 \neq 0$. We will consider two cases, small $q$ 's and large $q$'s.  To start with, we  analyse the character sum $\mathfrak{C}_{\pm}$  given in \eqref{character sum}. We have the following lemma which  is taken from \cite{munshi12}.
		\begin{lemma}  \label{charbound}
			Let $\mathfrak{C}_{\pm}$ be as in \eqref{character sum}. Then, for $n_2 \neq 0$,	we have 
			$$\mathfrak{C}_{\pm} \ll \frac{q_{1}^2 \, r (m,n_{1})}{n_{1}} \mathop{\sum \sum}_{\substack{d_{2} \mid (q_{2},  n_{1} q_{2}^{\prime}\mp mn_{2}) \\ d_{2}^{\prime} \mid (q_{2}^{\prime},  n_{1} q_{2} \pm m^{\prime} n_{2})}} \, d_{2} d_{2}^{\prime} \, .$$
		\end{lemma}
		\begin{proof}
			Let's recall from  \eqref{character sum} that
			\begin{align*}
				\mathfrak{C}_{\pm}=\mathop{\sum \sum}_{\substack{d|q \\ d^{\prime}|q^\prime}}dd^\prime\mu\left(\frac{q}{d}\right)\mu\left(\frac{q^\prime}{d^\prime}\right)\mathop{\sideset{}{^ \star}\sum_{\substack{\alpha \, {\rm mod} \, qr/n_1 \\ n_1\alpha\equiv-m \, {\rm mod} \, d}} \  \sideset{}{^ \star}\sum_{\substack{\alpha^\prime \, {\rm mod} \, q^\prime r/n_1 \\ n_1\alpha^\prime \equiv-m^\prime \, {\rm mod} \, d^\prime}}}_{\pm \bar{\alpha}q_2^\prime \mp\bar{\alpha}^\prime q_2\equiv -n_2  \, {\rm mod} \, q_1q_2q_2^\prime r/n_1}1.
			\end{align*}
			Using the Chinese Remainder theorem,  we observe that $\mathfrak{C}_{\pm}$ can be dominated by a product of two sums $ \mathfrak{C}_{\pm} \ll \mathfrak{C}_{\pm}^{(1)} \mathfrak{C}_{\pm}^{(2)}$,
			where
			$$\mathfrak{C}_{\pm}^{(1)} = \mathop{\sum \sum}_{\substack{d_{1}, d_{1}^{\prime} | q_{1}}} d_{1} d_{1}^{\prime}  \; \mathop{\sideset{}{^\star}{\sum}_{\substack{\beta \; {\rm mod} \; \frac{q_{1}r}{n_{1}} \\  n_{1} \beta \;  \equiv \; - m \; {\rm mod} \;  d_{1}}} \  \sideset{}{^\star}{\sum}_{\substack{\beta^{\prime} \; {\rm mod} \; \frac{q_{1}r}{n_{1}} \\  n_{1}   \beta^{\prime} \;  \equiv \; - m^{\prime} \; {\rm mod} \;  d_{1}^{\prime}}}}_{\pm \overline{\beta} q_{2}^{\prime} \mp \overline{\beta^{\prime}} q_{2} + n_{2} \; \equiv \; 0  \;\mathrm{mod}\; {q_{1}r }/{n_{1}} } \; 1 $$
			and
			$$\mathfrak{C}_{\pm}^{(2)} = \mathop{\sum \sum}_{\substack{d_{2} \mid q_{2} \\ d_{2}^{\prime} \mid q_{2}^{\prime}}} d_{2} d_{2}^{\prime}  \; \mathop{\sideset{}{^\star}{\sum}_{\substack{\beta \; {\rm mod }\; q_{2} \\  n_{1} \beta \;  \equiv \; - m \; {\rm mod} \;  d_{2}}} \, \sideset{}{^\star}{\sum}_{\substack{\beta^{\prime} \; {\rm mod} \; q_{2}^{\prime} \\  n_{1}   \beta^{\prime} \;  \equiv \; - m^{\prime}  \; {\rm mod} \;  d_{2}^{\prime}}}}_{\pm \overline{\beta} q_{2}^{\prime} \mp \overline{\beta^{\prime}} q_{2} + n_{2} \; \equiv \; 0 \;\mathrm{mod}\,  q_{2} q_{2}^{\prime} } \; 1.$$
			On analysing the second sum $\mathfrak{C}_{\pm}^{(2)}$,  we get  $\beta \equiv -m\bar{n_1}\,  {\rm  mod} \, d_{2}$ and $\beta^{\prime} \equiv \, -m^{\prime}\bar{n_1} \, {\rm  mod} \, d_{2}^{\prime}$,  as $(n_{1},q_{2}q_{2}^{\prime})=1$. Then using the congruence modulo $q_{2} q_{2}^{\prime}$, we conclude that
			
			$$\mathfrak{C}_{\pm}^{(2)} \ll  \mathop{\sum \sum}_{\substack{d_{2} \mid (q_{2},  n_{1} q_{2}^{\prime}\mp mn_{2}) \\ d_{2}^{\prime} \mid (q_{2}^{\prime},  n_{1} q_{2} \pm m^{\prime} n_{2})}} \, d_{2} d_{2}^{\prime}.$$
			In the first sum $\mathfrak{C}_{\pm}^{(1)}$, the congruence condition determines $\beta$ uniquely in terms of $\beta^\prime$, and hence 
			$$\mathfrak{C}_{\pm}^{(1)} \ll \mathop{\sum \sum}_{\substack{d_{1}, d_{1}^{\prime} | q_{1}}} d_{1} d_{1}^{\prime} \sideset{}{^\star}{\sum}_{\substack{\beta \; {\rm mod} \; {q_{1}r}/{n_{1}} \\  n_{1} \beta \;  \equiv \; - m \; {\rm mod} \;  d_{1}}} \, 1 \ll \frac{r \, q_{1}^2 \, (m,n_{1}) }{n_{1}} .$$
			Hence we have the lemma.
		\end{proof}
		\subsection{$S_r(N)$ for small  $q$} In this subsection, we will estimate  $S_r(N)$ for small values of  $q$. Let $\Omega_{\pm}^{\neq 0}$ denote the part of $\Omega_{\pm}$ (defined in \eqref{final omega}) which is complement to $\Omega_{\pm}^{ 0}$ ( contribution of $n_2 \neq 0$) and let $S_r^{\neq 0}(N)$ denote the part of $S_r(N)$ corresponding to $\Omega_{\pm}^{\neq 0}$.  We have   the following lemma.
		\begin{lemma}\label{small q bound}
			Let $\Omega_{\pm}^{\neq 0} $ and  $S_r^{\neq 0}(N)$ be as above. Then, for  $C \ll k^{1+\epsilon}$, we have
			\begin{align}\label{omage bound for small q}
				\Omega_{\pm}^{\neq 0} \ll  \frac{k^\epsilon r^2C^7(TN)^{1/2}} {n_{1}^2q_1 Q^2M_1N}\left(\frac{CM_1n_1}{q_1}+M_1^2\right).
			\end{align}
			Furthermore, let $S_{r,\, small}^{\neq 0}(N)$  denote the contribution of  $C \ll k^{1+\epsilon}$ to $S_r^{\neq 0}(N)$. Then we have
			\begin{align}\label{S small}
				S_{r, small}^{\neq 0}(N) \ll  r^{1/2}k^{3-\eta/2},
			\end{align}
			where 	$T=k^{1-\eta}$.
		\end{lemma}
		\begin{proof}
			On applying  \eqref{fou coef} to  \eqref{Omega}, we see that   $\Omega_{\pm}^{\neq 0}$ is dominated by
			$$k^{\epsilon} \sup_{\tilde{N} \ll N_0}\frac{\tilde{N}}{n_1^2}\mathop{\sum\,  \sum}_{q_2,\, q_2^{\prime}\sim C/q_1} \mathop{\sum\,  \sum}_{m,\, m^{\prime} \sim M_1} (|\lambda_f(m)|^2+|\lambda_{f}(m^\prime)|^2) \sum_{n_2 \in \mathbb{Z} -\{0\}}|\mathfrak{C}_{\pm}||\mathcal{J}_{\pm}|.$$
		 We  analyse  the expression corresponding to  $|\lambda_{f}(m^\prime)|^2$ only, since the calculations for  the other  expression is absolutely similar.  Thus,  on applying  Lemma \ref{charbound} and   Corollary \ref{final int bound},    we arrive at 
			\begin{align*}
			  \frac{k^\epsilon q_{1}^2 rC^4} {n_{1}^3Q^2M_1N}  \sup_{\tilde{N} \ll N_0}{\tilde{N}} \, \mathop{\sum \sum }_{\substack{q_2,\, q_{2}^{\prime} \sim \frac{C}{q_{1}}  } } \mathop{\sum \sum}_{\substack{d_{2} \mid q_{2} \\ d_{2}^{\prime} \mid q_{2}^{\prime}}} d_{2} d_{2}^{\prime} \, \mathop{ \mathop{\sum \ \sum \ \  \ \sum}_{m,\, m^{\prime} \sim M_{1} \ n_2 \in \mathbb{Z}-\{0\}}}_{\substack{ n_{1} q_{2}^{\prime} \mp m n_{2} \, \equiv \,  0 \, {\rm mod} \, d_{2} \\  n_{1} q_{2} \pm m^{\prime} n_{2} \, \equiv \,  0 \, {\rm mod} \, d_{2}^{\prime}}}  |\lambda_f(m^\prime)|^2(m,n_{1}).
			\end{align*}
			 Writing $q_2d_2$  and $q_2^{\prime}d_2^{\prime}$ in place of $q_2^{\prime}$ and $q_2^\prime$ respectively, we arrive at
			\begin{align}\label{omega nonzero}
			  \frac{k^\epsilon q_{1}^2 rC^4} {n_{1}^3Q^2M_1N}  \sup_{\tilde{N} \ll N_0}{\tilde{N}} \mathop{\sum \sum}_{d_{2}, d_{2}^{\prime} \ll C/q_{1} } d_{2} d_{2}^{\prime} \, \mathop{\sum \sum }_{\substack{q_{2} \sim \frac{C}{d_{2}q_{1}} \\ q_{2}^{\prime} \sim \frac{C}{d_{2}^{\prime} q_{1}}}}  \mathop{ \mathop{\sum \ \sum \    \sum}_{m,\,m^{\prime} \sim M_{1} \  1\leq|n_2|\ll N_2}}_{\substack{ n_{1} q_{2}^{\prime} d_{2}^{\prime}\mp m n_{2} \, \equiv \,  0 \, {\rm mod} \, d_{2} \\  n_{1} q_{2} d_{2}\pm m^{\prime} n_{2} \, \equiv \,  0 \, {\rm mod} \, d_{2}^{\prime}}} |\lambda_f(m^\prime) |^2 (m,n_{1}). 
			\end{align}
			Fixing the parameters $(n_2, q_2,q_2^\prime,d_2, d_2^\prime, m^\prime)$, we  count the number of $m$ as follows: 
			\begin{align}\label{m sum}
				\sum_{\substack{m\sim M_{1} \\ n_{1} q_{2}^{\prime} d_{2}^{\prime}\mp m n_{2} \, \equiv \,  0 \, {\rm mod} \, d_{2}}} (m,n_{1}) & = \sum_{\ell \mid n_{1}} \ell \,  \sum_{\substack{m \sim M_{1}/\ell \\ n_{1} q_{2}^{\prime} d_{2}^{\prime}\overline{\ell} \mp m n_{2} \, \equiv \,  0 \, {\rm mod} \, d_{2}}} 1   \\
				&= \sum_{\ell \mid n_{1}} \ell \, \left((d_2,q_2^\prime d_2^\prime,n_2)+\frac{M_1}{\ell d_2/(d_2,d_2^\prime q_2^\prime,n_2)}\right) \notag \\
				& \ll (d_{2}, d_2^\prime q_2^\prime\, ,n_{2}) \, \left(n_{1}+\frac{M_{1}}{d_{2}}\right) \notag,
			\end{align} 
		where  $\bar{\ell}$  is the inverse of $\ell $  modulo $d_2$ which follows from  the fact $(d_2,n_1)=1$.   On applying  \eqref{m sum} with the bound $ (d_{2} ,n_{2}) ( n_{1}+{M_{1}}/{d_{2}}) $ and then executing the sum over $q_2^\prime$ in \eqref{omega nonzero}, we arrive at 
				\begin{align}\label{sum over q2}
				\frac{k^\epsilon q_{1}^2 rC^4} {n_{1}^3Q^2M_1N} & \sup_{\tilde{N} \ll N_0}{\tilde{N}}  \mathop{\sum \sum}_{d_{2}, d_{2}^{\prime} \ll C/q_{1} } \frac{Cd_2}{q_1} \, \mathop{ \sum }_{\substack{q_{2} \sim \frac{C}{d_{2}q_{1}} }} \\
				 &\times  \mathop{ \mathop{\sum \ \ \ \ \sum }_{ 1\leq |n_2|\ll N_2,\,   m^\prime \sim M_{1} }}_{\substack{   n_{1} q_{2} d_{2}\pm m^{\prime} n_{2} \, \equiv \,  0 \, {\rm mod} \, d_{2}^{\prime}}} |\lambda_f(m^\prime) |^2 (d_2,n_2)\left(n_{1}+\frac{M_{1}}{d_{2}}\right) \notag. 
			\end{align}
			 We  now  count the number of   $(d_2, d_2^\prime, m^\prime)$  following the arguments in \cite[Section 6.1]{Lin-Mic-Saw}.
				\paragraph*{Case 1}  $ n_{1} q_{2} d_{2}\pm m^\prime n_{2} \, \equiv \,  0 \, {\rm mod} \, d_{2}^\prime $ but $n_{1} q_{2} d_{2}\pm m^\prime n_{2}  \neq 0$.  On switching the order of summations over $d_2^\prime$ and $m^\prime$, we see that the $d_2^\prime$-sum is bounded above by $d(|n_{1} q_{2} d_{2}\pm m^\prime n_{2} |)\ll k^\epsilon$, with $d(n)$ being the divisor function.  Thus \eqref{sum over q2} is bounded above  by 
					\begin{align*}
					\frac{k^\epsilon q_{1}^2 rC^4} {n_{1}^3Q^2M_1N} & \sup_{\tilde{N} \ll N_0}{\tilde{N}} \mathop{\sum }_{d_{2} \ll {C}/{q_1}} \frac{Cd_2}{q_1} \, \mathop{ \sum }_{\substack{q_{2} \sim \frac{C}{d_{2}q_{1}}  }} \\
					&\times\sum_{ 1\leq|n_2|\ll N_2}  \sum_{ m^\prime \sim M_{1}} |\lambda_f(m^\prime) |^2 (d_2,n_2) \left(n_{1}+\frac{M_{1}}{d_{2}}\right). 
				\end{align*}
				On applying the Ramanujan bound on average to the $m^\prime$-sum (see \eqref{RS for holomorphic}, \eqref{RS bound}) and executing the $n_2$-sum, we arrive at 
					\begin{align*}
					\frac{k^\epsilon q_{1}^2 rC^4} {n_{1}^3Q^2M_1N} & \sup_{\tilde{N} \ll N_0}{\tilde{N}} N_2M_1 \mathop{\sum }_{d_{2} \ll {C}/{q_1}} \frac{Cd_2}{q_1} \, \mathop{ \sum }_{\substack{q_{2} \sim \frac{C}{d_{2}q_{1}}  }} \left(n_{1}+\frac{M_{1}}{d_{2}}\right). 
				\end{align*}
			Now executing the remaining sums, we get the following expression
						\begin{align} \label{omega nonzero final }
				\frac{k^\epsilon  rC^6} {n_{1}^3Q^2M_1N} & \sup_{\tilde{N} \ll N_0}{\tilde{N}}  N_2   \left(\frac{Cn_1M_1}{q_1}+{M_{1}^2}\right). 
			\end{align}
	 On applying  the bounds  $N_2=k^{\epsilon}{CN^{1/3}r^{2/3}n_1}/{(q_1\tilde{N}^{2/3})}$ (see  \eqref{N2}) and $N_0 \ll k^\epsilon T^{3/2}\sqrt{N}r $ (see \eqref{N0}), we note that 
			\begin{align} \label{N0N2}
				\sup_{\tilde{N} \ll N_0} \tilde{N}N_2 \ll k^\epsilon\frac{Cr^{2/3}n_1}{q_1}(N\tilde{N})^{1/3} \ll k^\epsilon\frac{Cr^{2/3}n_1}{q_1}(NN_0)^{1/3} \ll  \frac{k^\epsilon rn_1}{q_1}(TN)^{1/2}C.
			\end{align}
		Thus, in Case 1, we get  the following bound for $	\Omega_{\pm}^{\neq 0}$
		\begin{align} \label{omega nonzero bound}
			 \frac{k^\epsilon r^2C^7(TN)^{1/2}} {n_{1}^2q_1 Q^2M_1N} \left(\frac{Cn_1M_1}{q_1}+M_1^2\right).
			\end{align}
			
			\paragraph*{Case 2} $n_{1} q_{2} d_{2}\pm m^\prime n_{2}  = 0$. On applying \eqref{m sum} with the bound 
			 $ (d_{2} ^\prime q_2^\prime,n_{2}) ( n_{1}+{M_{1}}/{d_{2}}) $  and switching some summations in   \eqref{omega nonzero}, we arrive  at 
				\begin{align}
				\frac{k^\epsilon q_{1}^2 rC^4} {n_{1}^3Q^2M_1N}  \sup_{\tilde{N} \ll N_0}{\tilde{N}} &\mathop{\sum \sum}_{d_{2}, d_{2}^{\prime} \ll C/q_{1} } d_{2} d_{2}^{\prime} \, \mathop{\sum \sum }_{\substack{ q_{2}^{\prime} \sim \frac{C}{d_{2}^{\prime} q_{1}}}} \sum_{m^\prime \sim M_1}  |\lambda_f(m^\prime) |^2 \\
				& \times \mathop{ \mathop{\sum \ \ \ \ \sum \ \  }_{  1\leq|n_2|\ll N_2 \   q_2 \sim C/d_2q_1 }}_{\substack{  n_{1} q_{2} d_{2}\pm m^{\prime} n_{2} =  0  }}  (d_2^\prime q_2^\prime,n_{2}) \left(n_{1}+\frac{M_{1}}{d_{2}}\right). \notag
			\end{align}
			Fixing  the tuple $(m^\prime, n_2,d_2)$, the number of $q_2$ turns out to be $O(k^\epsilon)$ (as $q_2 |m^\prime n_2$). Thus we arrive at 
				\begin{align*}
				\frac{k^\epsilon q_{1}^2 rC^4} {n_{1}^3Q^2M_1N}  \sup_{\tilde{N} \ll N_0}{\tilde{N}} &\mathop{ \sum}_{ d_{2}^{\prime} \ll C/q_{1} }  d_{2}^{\prime} \, \mathop{ \sum }_{\substack{ q_{2}^{\prime} \sim \frac{C}{d_{2}^{\prime} q_{1}}}} \sum_{m^\prime \sim M_1}  |\lambda_f(m^\prime) |^2 \\
				& \times \sum_{1 \leq |n_2| \ll N_2}  (d_2^\prime q_2^\prime,n_{2}) \mathop{\sum}_{\substack{d_2 \ll C/q_1 \\ d_2| m^\prime n_2}} \left(n_{1}d_2+{M_{1}}\right). 
			\end{align*}
			Now executing the sum over $d_2$, followed by the sum over $n_2$, $m^\prime$, $q_2^\prime$ and $d_2^\prime$, we see that the above expression is bounded above by
				\begin{align*}
				\frac{k^\epsilon  rC^6} {n_{1}^3Q^2M_1N} & \sup_{\tilde{N} \ll N_0}{\tilde{N}}  N_2   \left(\frac{Cn_1M_1}{q_1}+{M_{1}^2}\right). 
			\end{align*}
			 Now estimating $\tilde{N}N_2$ like Case 1, we get the first part of the  lemma. 
		We  will now prove 	 \eqref{S small}. Consider the second term of the right  hand side in \eqref{omage bound for small q}.  On  substituting it in place of $\Omega_{\pm}$ in   $S_r(N)$ in \eqref{S(N)after cauchy}, we arrive at 
			\begin{align*}
				& \mathop{\sup_{ \substack{M \ll M_1\ll M_0 \\ C \ll k^{1+\epsilon}}}}\frac{N^{5/3+\epsilon}}{QTr^{2/3}C^3}\sum_{\pm}\sum_{\frac{n_1}{(n_1,r)}\ll C}n_1^{1/3}\Theta^{1/2}\sum_{\frac{n_1}{(n_1,r)}|q_1|(n_1r)^\infty}\left( \frac{ r^2C^7(TN)^{1/2}M_1} {n_{1}^2q_1 Q^2N}\right)^{1/2} \\
				&\ll \mathop{\sup_{ \substack{ M \ll M_1 \ll M_0\\  C \ll k^{1+\epsilon}}}}\frac{N^{5/3+\epsilon}}{QTr^{2/3}C^3}\frac{r(TN)^{1/4}C^{7/2}M_1^{1/2}}{Q\sqrt{N}}\sum_{n_1\ll Cr}n_1^{-2/3}\Theta^{1/2}\sum_{\frac{n_1}{(n_1,r)}|q_1|(n_1r)^\infty}\frac{1}{q_1^{1/2}} \notag \\
				&\ll  \mathop{\sup_{ \substack{M \ll M_1 \ll M_0 \\ C \ll k^{1+\epsilon}}}}\frac{N^{5/3+\epsilon}}{QTr^{2/3}}\frac{r(TN)^{1/4}C^{1/2}M_1^{1/2}}{Q\sqrt{N}}\sum_{n_1\ll Cr}\frac{\sqrt{(n_1,r)}}{n_1^{7/6}}\Theta^{1/2}\\
				& \ll k^{\epsilon}r^{1/2}k^{3-\eta/2},
			\end{align*}
			where in the second last inequality, we used  
			$$ \sum_{n_1\ll Cr}\frac{\sqrt{(n_1,r)}}{n_1^{7/6}}\Theta^{1/2}\ll_{\pi, \epsilon} N_0^{1/6+\epsilon}$$
			from  \eqref{theta bound}, $C \ll k^{1+\epsilon}$, $N_0 \ll k^\epsilon r\sqrt{N}T^{3/2}$ and   $M_0 \ll k^{4+\epsilon}/N$ as $C \ll k^{1+\epsilon}$. 
			
			Let's now  consider the first term in the right hand side of \eqref{omage bound for small q}. We see that its contribution to $S_r(N)$ in  \eqref{S(N)after cauchy} is given by
			\begin{align*}
				& \mathop{\sup_{ \substack{M \ll M_1\ll M_0 \\ C \ll k^{1+\epsilon}}}}\frac{N^{5/3+\epsilon}}{QTr^{2/3}C^3}\sum_{\pm}\sum_{\frac{n_1}{(n_1,r)}\ll C}n_1^{1/3}\Theta^{1/2}\sum_{\frac{n_1}{(n_1,r)}|q_1|(n_1r)^\infty}\left( \frac{ r^2C^7(TN)^{1/2}C} {n_{1}q_1^2 Q^2N}\right)^{1/2} \\
				&\ll \mathop{\sup_{ \substack{ M \ll M_1 \ll M_0\\  C \ll k^{1+\epsilon}}}}\frac{N^{5/3+\epsilon}}{QTr^{2/3}C^3}\frac{r(TN)^{1/4}C^{7/2}C^{1/2}}{Q\sqrt{N}}\sum_{n_1\ll Cr}n_1^{-1/6}\Theta^{1/2}\sum_{\frac{n_1}{(n_1,r)}|q_1|(n_1r)^\infty}\frac{1}{q_1} \notag \\
				&\ll  \mathop{\sup_{ \substack{M \ll M_1 \ll M_0 \\ C \ll k^{1+\epsilon}}}}\frac{N^{5/3+\epsilon}}{QTr^{2/3}}\frac{r(TN)^{1/4}C}{Q\sqrt{N}}\sum_{n_1\ll Cr}\frac{{(n_1,r)}}{n_1^{7/6}}\Theta^{1/2} \\
				&	\ll k^{3-\eta/2}.
			\end{align*}
			In the second last inequality, we used the bound
			\begin{align*}
				\sum_{n_1\ll Cr}\frac{(n_1,r)}{n_1^{7/6}}\Theta^{1/2} \ll \left[\sum_{n_1 \ll Cr}\frac{(n_1,r)^2}{n_1}\right]^{1/2}\left[ \mathop{\sum \sum}_{n_{1}^{2} n_{2} \leq N_0} \frac{\vert \lambda_{\pi}(n_{1},n_{2})\vert ^{2} }{(n_1^2n_2)^{2/3}}\right]^{1/2}\ll_{\pi, \epsilon}r^{1/2}N_0^{1/6+\epsilon}.
			\end{align*}
			 Thus we have the lemma.
		\end{proof}
		\subsection{Estimates for generic $q$} Now we  tackle the case when $C \gg k^{1+\epsilon}$ and $n_2 \neq 0$.
		Let $S_{r, generic}^{\neq 0}(N)$  denote the part of $S_r^{\neq 0}(N)$ which is complement to $S_{r, small}^{\neq 0}(N)$ (i.e.,  the contribution of  $C \gg k^{1+\epsilon}$) and $n_2 \neq 0 $ to $S_r(N)$. We have the following lemma.
		\begin{lemma}\label{SN generic bound final}
			Let $	S_{r, generic}^{\neq 0}(N) $ be as above.   Then  we have
			\begin{align}\label{S(N) generic bound}
				S_{r, generic}^{\neq 0}(N) \ll  N^{1/2}k^{3/2-1/6+3\eta/4}.
			\end{align}
		\end{lemma}
		\begin{proof}
			Let's recall from the analysis of $\Omega_{\pm}^{\neq 0}$ in the proof of  Lemma \ref{small q bound} that   (see \eqref{omega nonzero final }) 
			\begin{align}
				\Omega_{\pm}^{\neq 0}\ll 	 \frac{k^\epsilon  rC^6} {n_{1}^3Q^2M_1N}  \sup_{\tilde{N} \ll N_0}{\tilde{N}}N_2 \left(\frac{Cn_1M_1}{q_1}+M_1^2\right).
			\end{align}
			%
			%
			To get this, we used the bound  $\mathcal{J}_{\pm} \ll  {k^\epsilon C^4}/{(Q^2M_1N)}$. For $C \gg k^{1+\epsilon}$, we have a better bound for $\mathcal{J}_{\pm}$  (see Corollary \ref{final int bound}). In fact, we have 
			\begin{align}\label{generic int bound}
				\mathcal{J}_{\pm} \ll \frac{k^\epsilon C^2}{Q^2}\frac{Cr^{1/3}k^{2/3}}{k^2(N\tilde{N})^{1/3}} \asymp  \frac{k^\epsilon C^4}{Q^2M_1N}\frac{Cr^{1/3}k^{2/3}}{(N\tilde{N})^{1/3}},
			\end{align}
			where we used  $\sqrt{M_1N}/C \asymp k$ for $C \gg k^{1+\epsilon}$. Thus, on applying the above bound, we see that 
			\begin{align}
				\Omega_{\pm}^{\neq 0}\ll 	 \frac{k^\epsilon  rC^6} {n_{1}^3Q^2M_1N}\times Cr^{1/3}k^{2/3}\times   \sup_{\tilde{N} \ll N_0}\frac{{\tilde{N}}N_2}{(N\tilde{N})^{1/3}} \left(\frac{Cn_1M_1}{q_1}+M_1^2\right).
			\end{align}
			Recall from \eqref{N0N2} that  
			\begin{align}
				\sup_{\tilde{N} \ll N_0} \frac{{\tilde{N}}N_2}{(N\tilde{N})^{1/3}} \ll k^\epsilon\frac{Cr^{2/3}n_1}{q_1},
			\end{align}
			and 	$$\sup_{\tilde{N} \ll N_0} \frac{{\tilde{N}}N_2}{(N\tilde{N})^{1/3}} =\frac{{{N_0}}N_2}{(N{N_0})^{1/3}}.$$
			Thus  we see that 
			\begin{align}
				\Omega_{\pm}^{\neq 0}\ll 	 \frac{k^\epsilon  rC^6} {n_{1}^3Q^2M_1N}\times Cr^{1/3}k^{2/3}\times   \frac{{{N_0}}N_2}{(N{N_0})^{1/3}}\left(\frac{Cn_1M_1}{q_1}+M_1^2\right).
			\end{align}
			On comparing it with  \eqref{omega nonzero bound}, we observe  that we have an extra   factor $$\frac{Cr^{1/3}k^{2/3}}{r^{1/3}(NT)^{1/2}} \ll \frac{Qk^{2/3}}{(NT)^{1/2}}=k^{\epsilon+\eta-1/3}$$  
			in this case. 
			Hence,  taking it into account,  we get  
			\begin{align} \label{omeg}
				\Omega_{\pm}^{\neq 0}
				&	\ll  \frac{k^\epsilon r^2C^7(TN)^{1/2}} {n_{1}^2q_1 Q^2M_1N}\times k^{\eta-1/3}\left(\frac{Cn_1M_1}{q_1}+M_1^2\right).
			\end{align}
			Note that  
			\begin{align*}
				\frac{Cn_1}{q_1}+M_1 \ll \frac{Qn_1}{q_1}+M_0 \ll  \frac{n_1k^\epsilon}{q_1}\sqrt{\frac{N}{T}}+\frac{Q^2k^{2+\epsilon}}{N} 
				& \ll (n_1,r)k^\epsilon \sqrt{\frac{N}{T}}+\frac{k^{2+\epsilon}}{T} 
				& \ll \frac{k^{2+\epsilon}}{T},
			\end{align*}
			where we used $M_0\ll Q^2k^{2+\epsilon}/N$, $Nr^2 \ll k^{3+\epsilon}$, $Q=k^\epsilon\sqrt{N/T}$, $T \ll k$ and $n_1/q_1 \leq (n_1,r)$.
			Thus, on plugging the above bound into \eqref{omeg},  we get
			\begin{align} 
				\Omega_{\pm}^{\neq 0}
				&	\ll  \frac{k^\epsilon r^2C^7(TN)^{1/2}} {n_{1}^2q_1 Q^2N}\times k^{\eta-1/3} \times\frac{k^{2+\epsilon}}{T} \notag.
			\end{align}
			On substituting   the above bound in place of $\Omega_{\pm}$ in \eqref{S(N)after cauchy}, we see that   $	S_{r, generic}^{\neq 0}(N) $ is dominated by 
			\begin{align*}
				& \mathop{\sup_{ \substack{ C \ll Q}}}\frac{N^{5/3+\epsilon}}{QTr^{2/3}C^3}\sum_{\pm}\sum_{\frac{n_1}{(n_1,r)}\ll C}n_1^{1/3}\Theta^{1/2}\sum_{\frac{n_1}{(n_1,r)}|q_1|(n_1r)^\infty}\left( \frac{ r^2C^7(TN)^{1/2}} {n_{1}^2q_1 Q^2N}\right)^{1/2} \times \frac{k^{5/6+\eta/2}}{\sqrt{T}}\\
				&\ll \mathop{\sup_{ \substack{   C \ll Q}}}\frac{N^{5/3+\epsilon}}{QTr^{2/3}C^3}\frac{r(TN)^{1/4}C^{7/2}}{Q\sqrt{N}}\sum_{n_1\ll Cr}n_1^{-2/3}\Theta^{1/2}\sum_{\frac{n_1}{(n_1,r)}|q_1|(n_1r)^\infty}\frac{1}{q_1^{1/2}} \times \frac{k^{5/6+\eta/2}}{\sqrt{T}} \notag \\
				&\ll  \mathop{\sup_{ \substack{ C \ll Q}}}\frac{N^{5/3+\epsilon}}{QTr^{2/3}}\frac{r(TN)^{1/4}C^{1/2}}{Q\sqrt{N}}\sum_{n_1\ll Cr}\frac{\sqrt{(n_1,r)}}{n_1^{7/6}}\Theta^{1/2} \times \frac{k^{5/6+\eta/2}}{\sqrt{T}} \\
				&\ll  N^{1/2}k^{3/2-1/6+3\eta/4}.
			\end{align*}
			Hence the lemma follows. 
		\end{proof}
		\subsection{Estimates for the error term}\label{estimates for the error term} In this subsection, we give estimates for  $S_r(N)$ corresponding to the non-generic case  $n_2^\star N \ll k^\epsilon$ (see  Lemma \ref{n*}). Recall from \eqref{S3 for nongeneric}  that, if $n_2^\star N=n_1^2n_2N/(q^3r) \ll k^\epsilon$, then we have 
		\begin{align}  \label{S3 for error}
			\mathrm{S_3}
			=q  \sum_{\pm} \sum_{n_{1}|qr} \sum_{n_{2}=1}^{\infty}  \frac{\lambda_{\pi}(n_1,n_2)}{n_{1} n_{2}} S\left(r \bar{a}, \pm n_{2}; qr/n_{1}\right) G_{\pm} \left(n_2^\star\right),
		\end{align} 
		where $G_{\pm}(n_2^\star)$ is as  defined  in  \eqref{gl3 integral transform 2 }. On plugging  \eqref{S3 for error} and \eqref{S2 final} in place of $\mathrm{{S_3}}$ and $\mathrm{{S_2}}$ respectively into \eqref{SN after dfi} we arrive at  
		\begin{align} \label{SN after error}
			&\frac{2\pi i^kN^{1-it}}{QT} \sum_{1\leq q\leq Q}\,\frac{1}{q}   \sum_{\pm} \sum_{n_{1}|qr}  \sum_{ n_{2}\ll \frac{q^3rk^\epsilon}{ n_1^2 N}}  \frac{\lambda_{\pi}(n_1,n_2)}{n_{1} n_{2}} \\
			& \times   \sum_{M\leq m \leq M_0}\,\lambda_f(m) \, \mathcal{C_{\pm}}(...)\, \mathrm{I}_4(q,m,n_1^2n_2) 
			+	O(k^{-2020}),\notag
		\end{align}
		where 
		\begin{align}\label{character}
			\mathcal{C_{\pm}}(. . .)&:=\sideset{}{^\star}\sum_{a \, {\rm mod} \, q}S(r \bar a, \pm n_2; qr/n_1)e\left(\frac{\bar{a}m}{q}\right) \\
			&=\sum_{d|q}d\mu\left(\frac{q}{d}\right)\sideset{}{^ \star}\sum_{\substack{\alpha \, {\rm mod} \, qr/n_1 \\ n_1\alpha\equiv-m \, {\rm mod} \, d}}e\left(\pm\frac{\bar{\alpha}n_2}{qr/n_1}\right) \notag \\
			&\ll (n_1,m,q)\left(q+\frac{qr}{n_1}\right) \ll  \sqrt{(n_1,m)}  \sqrt{(n_1,q)}\left(q+\frac{qr}{n_1}\right), \notag
		\end{align}  
		and 
		\begin{align*}
			\mathrm{I}_4(q,m,n_1^2n_2)=\int_{\mathbb{R}}\,W(x/Q^\epsilon)\,\int_{\mathbb{R}}\,V\left(\frac{t}{T}\right) \,{g(q,x)}\,	\mathrm{I_2}(m,\,q,\,x) \,G_{\pm} \left(n_2^\star\right)\,\mathrm{d}t \, \mathrm{d}x,
		\end{align*}
		with 
		\begin{align*}
			\mathrm{I_2}(m,\,q,\,x)=\int_0^{\infty}U(y) y^{-it}e\left(\frac{-Nxy}{qQ}\right)J_{k-1}\left(\frac{4\pi \sqrt{mNy}}{q}\right)\mathrm{d} y,
		\end{align*}
		and 
		\begin{align}\label{contour}
			G_{\pm}(n_2^\star) &=\frac{1}{2 \pi i} \int_{(\sigma)} (n_2^\star)^{-s} \, \gamma_{\pm}(s)\tilde{g}(-s)  \, \mathrm{d}s  \\
			& =\frac{N^{it}}{2 \pi } \int_{-\infty}^{\infty}  \frac{\gamma_{\pm}(\sigma+i\tau)}{ (n_2^\star N)^{\sigma+i\tau}} \int_{0}^{\infty}\, e\left(\frac{z_1Nx}{qQ}\right)V\left(z_1\right)z_1^{-\sigma-i\tau+it}\, \frac{\mathrm{d}\, z_1}{z_1} \, \mathrm{d}\tau, \notag
		\end{align}
		where $\sigma > -1 + \max \{-\Re({\alpha}_{1}), -\Re({ \alpha}_{2}), -\Re({\alpha}_{3})\}$.
		On analysing the $x$-integral and the  $t$-integral following  Lemma \ref{restriction on u }, we get the  following restriction $$|z_1-y|\ll k^\epsilon q/QT.$$
		Thus, on replacing $z_1$ by $y+u$ with $u \ll k^\epsilon q/QT $, we essentially  arrive at
		\begin{align*}
			\mathrm{I}_4(q,m,n_1^2n_2)=\frac{1}{2\pi} \int_{-\infty}^{\infty} \frac{\gamma_{\pm}(\sigma+i\tau)}{ (n_2^\star N)^{\sigma+i\tau}}\int_{\mathbb{R}}V\left(\frac{t}{T}\right) N^{it}\,  \int_{u \ll \frac{k^\epsilon q}{QT}}\, \mathrm{I}_u	\, \mathrm{I_5}(m,\,q,\,u, \tau)\, \mathrm{d}u \,\mathrm{d}t \mathrm{d}\tau,
		\end{align*}
		where 
		\begin{align*}
			\mathrm{I}_{u}=\int_{\mathbb{R}}W(x/Q^{\epsilon})g(q,x)e\left(\frac{Nxu}{qQ}\right)\, \mathrm{d}x,
		\end{align*}
		and 
		$$ \mathrm{I_5}(m,\,q,\,u, \tau)=\int_{0}^{\infty}U_{t,u,\tau}(y) y^{-i\tau}J_{k-1}\left(\frac{4\pi \sqrt{mNy}}{q}\right)\,\mathrm{d}y, $$
		with $U_{t,u,\tau}(y)=U(y)y^{-\sigma}(1+u/y)^{-\sigma -i\tau+it}.$ On analysing   $ \mathrm{I_5}(m,\,q,\,u, \tau)$ like  $	\mathrm{I}_{\pm}(m,\tilde{N}w,q)$ (see   Lemma \ref{y-integral bound}), we get 
		$$ \mathrm{I_5}(m,\,q,\,u, \tau) \ll \frac{k^\epsilon q^{1/2}}{{(mN)}^{1/4}}.$$
		We now move the contour $\sigma$ in \eqref{contour}  to the left up to $\sigma =-5/2$ passing through the poles given by 
		$$\frac{1+\sigma +\Re{\alpha_{i} +\ell}}{2}=0 \iff \sigma =-1-\Re{\alpha_i}-\ell.$$ 
		Thus, on treating the $u$ and $t$-integral trivially,  we get 
		\begin{align*}
			\mathrm{I}_4(q,m,n_1^2n_2)	 \ll {(n_2^\star N)^{5/2}}& \frac{k^{\epsilon}q^{3/2}}{Q(mN)^{1/4}} \int_{-\infty}^{\infty} |\gamma_{\pm}(-5/2+i\tau)| \mathrm{d}\tau  \\
			&+\frac{k^{\epsilon}q^{3/2}}{Q(mN)^{1/4}}+\sum_{\ell=0,1} \sum_{i=1}^3(n_2^\star N)^{1+\ell+\Re \alpha_i}.
		\end{align*}
		Now using the Stirling  bound 
		$$|\gamma_{\pm}(-5/2+i\tau)| \ll (1+|\tau|)^{3(-5/2+1/2)}=(1+|\tau|)^{-6},$$ 
		we arrive at 
		\begin{align*}
			\mathrm{I}_4(q,m,n_1^2n_2) \ll\frac{k^{\epsilon}q^{3/2}}{Q(mN)^{1/4}} \left((n_2^\star N)^{5/2}+\sum_{\ell=0,1}\sum_{i=1}^3(n_2^\star N)^{1+\ell+\Re \alpha_i}\right).
		\end{align*}
		Note that $(n_2^\star N)^{5/2}=(n_2^\star N)^{1/4+9/4} \ll k^{\epsilon}(n_2^\star N)^{1/4} $, and 
		$$\sum_{i=1}^3(n_2^\star N)^{1+\ell+\Re \alpha_i}=\sum_{i=1}^3(n_2^\star N)^{1/2+\beta_i} \ll k^{\epsilon}(n_2^\star N)^{1/2} \ll k^{\epsilon}(n_2^\star N)^{1/4} $$
		as   $1+\ell +\Re \alpha_i =1/2+\beta_i$ for some $\beta_i >0$. Thus we get 
		\begin{align}
			\mathrm{I}_4(q,m,n_1^2n_2) \ll \frac{k^{\epsilon}q^{3/2}}{Q(mN)^{1/4}} \left(n_2^\star N \right)^{1/4}=\frac{k^{\epsilon}q^{3/4}(n_1^2n_2)^{1/4}}{Qm^{1/4}r^{1/4}}.
		\end{align}
		%
		%
		%
			Thus, on plugging the above bound and  the bound \eqref{character} for $	\mathcal{C_{\pm}}(. . .)$  into \eqref{SN after error} and then  estimating the sum over $m$ using the Ramanujan bound on average, we see that \eqref{SN after error} is dominated by
			\begin{align} 
				&  \sum_{1\leq q\leq Q}\frac{{NM_0^{3/4}}}{Q^{2}Tr^{1/4}}\,  \sum_{n_{1}|qr} \sum_{ n_{2}\ll \frac{q^3rk^\epsilon}{ n_1^2 N}}  \frac{|\lambda_{\pi}(n_1,n_2)|}{n_{1} n_{2}}(n_1^2n_2)^{1/4} \sqrt{(n_1,q)}\left(1+\frac{r}{n_1}\right).
			\end{align}
			We estimate the sum over  $n_1$ and $n_2$ as follows:
			\begin{align*}
				&\sum_{n_{1}|qr}  \sum_{ n_{2}\ll \frac{q^3rk^\epsilon}{ n_1^2 N}}   \frac{ |\lambda_{\pi}(n_1,n_2)|}{n_{1} n_{2}}(n_1^2n_2)^{1/4} \sqrt{(n_1,q)} \left(1+\frac{r}{n_1}\right)\\
				& \ll  \sum_{n_{1}|qr}  \sum_{ n_{2}\ll \frac{q^3rk^\epsilon}{ n_1^2 N}}   {|\lambda_{\pi}(n_1,n_2)|\frac{r}{\sqrt{n_2}}} \\
				&\ll 	\left( \mathop{\sum \sum}_{n_{1}^{2} n_{2} \ll k^\epsilon q^3r/N}|\lambda_{\pi}(n_1,n_2)|^2 \right)^{1/2}   \left( \sum_{n_{1}|qr} \sum_{n_{2}=1}^{\infty}\frac{r^2}{n_2} \right)^{1/2} \\
				& \ll \frac{q^{3/2}{r^{3/2}}}{\sqrt{N}}.
			\end{align*}
			Hence  the contribution of   the terms  $n_1^2n_2N/(q^3r) \ll k^\epsilon$  to $S_r(N)$  is dominated by
						\begin{align} \label{SN for error term}
			 &  k^{\epsilon}\sum_{1\leq q\leq Q}\frac{{NM_0^{3/4}}}{Q^{2}Tr^{1/4}}\, \frac{q^{3/2}{r^{3/2}}}{\sqrt{N}}\ll  \sqrt{N}k^{1+2\eta+3\eta/8 +\epsilon},
			\end{align}
			where we used $M_0 \ll k^{2+\epsilon}/T$ and $Nr^2 \ll k^{3+\epsilon}$.
			\section{\bf Conclusion: Proof of Theorem \ref{Main}} We now pull together the bounds from Lemma \ref{SN bound for zero frq}, Lemma \ref{small q bound}, Lemma \ref{SN generic bound final} and  \eqref{SN for error term} to get that 
			\begin{align*}
				\frac{S_r(N)}{N^{1/2}k^{3/2+\epsilon}}\ll k^{-1/2+2\eta +3\eta/8}+r^{1/2}k^{-\eta/2}+r^{1/2}\frac{k^{3/2-\eta/2}}{N^{1/2}}+k^{-1/6+3\eta/4}.
			\end{align*}
			Using $k^{3-\theta} \ll Nr^2\ll k^{3+\epsilon}$ and $r \ll k^{\theta}$, we further get
			\begin{align*}
				\frac{S_r(N)}{N^{1/2}k^{3/2+\epsilon}}\ll k^{-1/2+2\eta+3\eta/8}+k^{\theta/2-\eta/2}+k^{2\theta-\eta/2}+k^{-1/6+3\eta/4}.
			\end{align*}
			Hence to get subconvexity, we need  all of the above exponents to be negative. So the  first and the  third term gives  $4/19>\eta >4\theta$, and consequently the  third and the fourth terms dominate the rest. Thus the above bound reduces to 
			$$\frac{S_r(N)}{N^{1/2}k^{3/2+\epsilon}}\ll k^{2\theta-\eta/2}+k^{-1/6+3\eta/4}.$$
			The optimal choice for $\eta$ is given by $\eta=8\theta/5+2/15$. On  plugging this in Lemma \ref{AFE}, we get 
			$$L(1/2, \pi \times f) \ll k^{3/2+ 6\theta/5-1/15+\epsilon}+k^{3/2-\theta/2+\epsilon}, $$
			and with the optimal choice $\theta =2/51$, we obtain the bound given in Theorem \ref{Main}.


\begin{thebibliography}{999}
				
				\bibitem{BR}
				J. Bernstein and A. Reznikov: \emph{Periods, subconvexity of $L$-functions and representation theory}, J. Differential Geom. {\bf 13} (2005), 129--141.
				
				\bibitem{BKY} V. Blomer, R. Khan and M. Young: \emph{Distribution of mass of holomorphic cusp forms}, Duke Math. J. {\bf 162}(2013),  no. 14, 2609--2644.
				\bibitem{Blomer} V. Blomer and J. Buttcane: \emph{On the subconvexity problem for $L$-functions on $GL(3)$},  Ann. Sci. Éc. Norm. Supér. (4)  {\bf 53} (2020), no. 6, 1441--1500. 
				
				
				
				\bibitem{Burg}
				D. A. Burgess: \emph{On character sums and L-series. II}, Proc. London. Soc. {\bf 13} (1963),  524--536.
				
				
				\bibitem{CI}
				B. Conrey and  H.  Iwaniec: \emph{The cubic moment of central values of automorphic $L$-functions}, Ann. of Math. { \bf 151} (2000), 1175--1216.
				
				\bibitem{DFI}
				W. Duke,  J.  Friedlander and  H.  Iwaniec: \emph{Bounds for automorphic $L$-functions}, Invent. Math. { \bf 112} (1993), 1--8.
				
				
				\bibitem{DFI-2} 
				W. Duke,  J. Friedlander and  H. Iwaniec: \emph{Bounds for automorphic $L$-functions. II}, Invent. Math. { \bf 115} (1994), 219--239.
				
				\bibitem{DFI-2.1} 
				W. Duke,  J. Friedlander and  H. Iwaniec: \emph{Erratum: Bounds for automorphic $L$-functions. II}, Invent. Math. { \bf 140} (2000), 227--242.
				
				
				\bibitem{gold} D. Goldfeld: \emph{Automorphic forms and $L$-functions for the group $GL(n,\mathbb{R})$}, Cambridge Stud. Adv. Math. , {\bf 99}, Cambridge Univ. Press, Cambridge, 2006.
				
				\bibitem{gl3gl2ood}
				A. Good: \emph{The square mean of Dirichlet series associated with cusp forms}, Mathematika {\bf 29} (1982),  278--295.
				\bibitem{HM} G. Harcos and P. Michel: \emph{The subconvexity problem for Rankin-Selberg L-function and equidistribution of Heegner points. II}, Invent. Math. {\bf 163} (2006), 581--655.
				
				\bibitem{BingH}
				B. Huang: \emph{On the Rankin-Selberg problem},  Math. Ann. { \bf 381} (2021), no. 3-4,  1217--1251. 
				
				\bibitem{Huxley}
				M. N. Huxley: \emph{Area, lattice points, and exponential sums}, London Mathematical Society Monographs. New Series, {\bf 13}, The Clarendon Press, Oxford University Press, New York, 1996. Oxford Science Publications. 
				
				\bibitem{iwa} H. Iwaniec: \emph{The spectral growth of automorphic $L$-functions}, J. Reine angew. Math. {\bf 428} (1992), 139--159.
				
				\bibitem{i+k} H. Iwaniec and E. Kowalski: \emph{Analytic number theory}, American Mathematical society colloquium publications,  {\bf 53}, American Mathematical Society, Providence, RI, 2004.
				
				\bibitem{KMV} E. Kowalski, P. Michel and J. Vanderkam:\emph{ Rankin-Selberg L-functions in the level aspect}, Duke Math. J. {\bf 114} (2002), 123--191.
%
				
				\bibitem{KMS}
				S. Kumar,  K. Mallesham and  S.  Singh: \emph{Sub-convexity bound for $GL(3) \times GL(2)$ $L$-functions: $GL(3)$-spectral aspect}, https://arxiv.org/abs/2006.07819, 2020.
				
				\bibitem{Langer}
				R. E. Langer: \emph{On the asymptotic solutions of ordinary differential equations, with an application to the Bessel functions of  large order},  Transactions  of the  American Mathematical Society,  {\bf 33} (1931), no. 1, 23--64.
				\bibitem{Mic} P. Michel:  \emph{The subconvexity problem for Rankin-Selberg L-function and equidistribution of Heegner points}, Ann. of Math. {\bf 160} (2004), 185--236.
				
				\bibitem{MV}
				P. Michel and A. Venkatesh: \emph{The subconvexity problem for $GL(2)$}. Publ. Math. IHES {\bf 111} (2010),  171--280.
				
				
				\bibitem{miller-schmid} S. D. Miller and  W. Schmid: \emph{Automorphic distributions, L-functions, and Voronoi summation for  $GL(3)$}, Ann. of Math. {\bf 164} (2006), 423--488.
				
				
				\bibitem{LLY} Y. Lau, J. Liu, and Y. Ye: \emph{A new bound $k^{2/3+\epsilon}$ for  Rankin-Selberg $L$-functions for Hecke congruence subgroups}, IMRP Int. Math. Res. Pap. (2006), Art. ID 35090, 78 pp. 
				
				\bibitem{Li*} X. Li: \emph{The central value of the Rankin-Selberg $L$-functions}. Geom. Funct. Anal. {\bf 18} (2009), 1660--1695.
				
				\bibitem{Li} X. Li: \emph{Bounds for $GL(3) \times GL(2)$ $L$-functions and $GL(3)$ $L$-functions}, Annals of Math. {\bf 173} (2011), 301--336.
				
						\bibitem{Lin-Mic-Saw} Y. Lin, P. Michel and W. Sawin, \emph{Algebraic twists of $GL(3) \times GL(2)$ $L$-functions},   Amer. Journ. Math.  AJM 145.2, April 2023. 
						
						
						
				\bibitem{munshi1} R. Munshi: \emph{The circle method and bounds for L-functions-III}, Journal of the American Mathematical Society {\bf 28} (2015), 913--938.
				
				
				\bibitem{annals}
				R. Munshi: \emph{The circle method and bounds for $L$-functions-IV: subconvexity for twist of $GL(3)$ $L$-functions}, Annals of  Math. {\bf 182} (2015),  617--672.
				
				\bibitem{munshi12} R. Munshi: \emph{Subconvexity for $GL(3) \times GL(2)$ $L$-functions in $t$-aspect},  J. Eur. Math. Soc. {\bf 24}(2022), no. 5, 1543--1566.
				
				\bibitem{rankin}
				R. Rankin: \emph{The vanishing of Poincar\'e series}, Proc. Edin. Math. Soc. {\bf 23} (1980), no. 2, 151--161.
				\bibitem{Sar} P. Sarnak: \emph{Estimates for Rankin-Selberg $L$-functions and quantum unique ergodicity}, J. Funct. Anal. {\bf 184} (2001), 419--453.
				
				
				\bibitem{Venk} A.  Venkatesh:  \emph{ Sparse equidistribution problems, period bounds and subconvexity}, Annals of Math. (2) {\bf 172} (2010), no. 2, 989--1094.
				
	
				
					\bibitem{prahlad} 
				P. Sharma, \emph{Subconvexity for $GL(3) \times GL(2)$ twists}. With an appendix by Will Sawin. Adv. Math. {\bf 404}(2022), part B,  Paper No. 108420, 47 pp.
				
				
				
				\bibitem{Weyl}
				H. Weyl: \emph{Zur absch\"{a}tzung von $\zeta(1+it)$}, Math. Z. {\bf 10} (1921), 88--101.
				
				
			\end{thebibliography}
		\end{document}